\documentclass[11pt]{article}
\usepackage[margin=2.0cm]{geometry} %Changes the margins of the paper.
\usepackage{graphicx} %Simple graphics.
\usepackage[colorlinks=false]{hyperref} % For making hyperlink references.
\usepackage{amssymb} %Symbols from AMS.
\usepackage{amsfonts} %Fonts from AMS.
\usepackage{amsmath,amscd} %Diagrams etc.
\usepackage{amsthm}  %Enviroment for making theorems, definitions etc.
\usepackage{mathtools} % For inclusion arrows etc.
\usepackage{titlesec} %Section and subsection fonts
\usepackage{csquotes} %Quotation marks
\usepackage[shortlabels]{enumitem} %changes line spacing within lists
\usepackage{subfiles} %Divide document into subfiles
\usepackage{tikz-cd} % Diagrams
\usepackage{wrapfig} % Textflow around image
\usepackage{cutwin}
\usepackage{tabularx} % To get several equation on the same line
\usepackage{hyperref}
\hypersetup{
     colorlinks   = true,
     citecolor    = green
}

\usepackage{nomencl}% List of symbols

\makenomenclature

\usepackage{afterpage}

\setlist[itemize]{itemsep=0mm} %changes line spacing within lists
\numberwithin{equation}{section} %Numberings of equations within sections.
\titleformat*{\section}{\Large \scshape\center} %Section fonts
\titleformat*{\subsection}{\fontsize{14}{14} \sffamily} %Subsection fonts

\theoremstyle{plain}
\newtheorem{theorem}{Theorem}[section]
\newtheorem*{theorem*}{Theorem} %Unnumbered theorem
\newtheorem{lemma}[theorem]{Lemma}
\newtheorem{proposition}[theorem]{Proposition}
\newtheorem{corollary}[theorem]{Corollary}

\theoremstyle{definition}
\newtheorem{definition}[theorem]{Definition}
\newtheorem{example}[theorem]{Example}

\newcommand{\Addresses}{{% additional braces for segregating \footnotesize
  \bigskip
  \footnotesize

  \textsc{Department of Mathematical Sciences, Norwegian University of Science and Technology,\\ 7491 Trondheim, Norway.}\par\nopagebreak
  \textit{E-mail address}: \texttt{eirik.berge@ntnu.no}
}}

\theoremstyle{remark}
\newtheorem*{remark}{Remark}

\begin{document}
\pagenumbering{gobble}
\title{\Huge{$\alpha$-modulation spaces for step two \\ stratified Lie groups}}
%\title{\Huge{Generalized $\alpha$-Modulation Spaces on Step Two Stratified Lie Groups with Rational Structure Constants}}
\author{Eirik Berge}
\date{}
\maketitle

\pagenumbering{arabic}

\begin{abstract}
    We define and investigate $\alpha$-modulation spaces $M_{p,q}^{s,\alpha}(G)$ associated to a step two stratified Lie group $G$ with rational structure constants. This is an extension of the Euclidean $\alpha$-modulation spaces $M_{p,q}^{s,\alpha}(\mathbb{R}^n)$ that act as intermediate spaces between the modulation spaces ($\alpha = 0$) in time-frequency analysis and the Besov spaces ($\alpha = 1$) in harmonic analysis. We will illustrate that the the group structure and dilation structure on $G$ affect the boundary cases $\alpha = 0,1$ where the spaces $M_{p,q}^{s}(G)$ and $\mathcal{B}_{p,q}^{s}(G)$ have non-standard translation and dilation symmetries. Moreover, we show that the spaces $M_{p,q}^{s,\alpha}(G)$ are non-trivial and generally distinct from their Euclidean counterparts. Finally, we examine how the metric geometry of the coverings $\mathcal{Q}(G)$ underlying the $\alpha = 0$ case $M_{p,q}^{s}(G)$ allows for the existence of geometric embeddings \[F:M_{p,q}^{s}(\mathbb{R}^k) \longrightarrow{} M_{p,q}^{s}(G),\] as long as $k$ (that only depends on $G$) is small enough. Our approach naturally gives rise to several open problems that is further elaborated at the end of the paper.
\end{abstract}
\section{Introduction}
\label{sec: Introduction}

The modulation spaces $M_{p,q}^{s}(\mathbb{R}^n)$ in time-frequency analysis and the (inhomogeneous) Besov spaces $\mathcal{B}_{p,q}^{s}(\mathbb{R}^{n})$ in harmonic analysis are invaluable in their own fields. They are connected by the existence of a one-parameter family of Banach spaces $M_{p,q}^{s,\alpha}(\mathbb{R}^{n})$ where $0 \leq \alpha \leq 1$ such that the aforementioned spaces are the boundary cases $\alpha = 0$ and $\alpha = 1$. It was in the Ph.D. thesis \cite{GrobnerThesis} that the \textit{$\alpha$-modulation spaces} $M_{p,q}^{s,\alpha}(\mathbb{R}^n)$ were first introduced and they have subsequently been investigated for a plethora of reasons: The $\alpha$-modulation spaces $M_{p,q}^{s,\alpha}(\mathbb{R}^{n})$ are suitable spaces for studying diverse questions such as boundedness of psuedodifferential operators \cite{BN2}, embedding questions \cite{Embedding_alpha_mod, Felix_main}, and Banach frame expansions \cite{BN1}. Moreover, the spaces $M_{p,q}^{s,\alpha}(\mathbb{R}^{n})$ have found applications in non-linear approximation theory \cite{BN4} and for studying the Cauchy problem for nonlinear Schr\"{o}dinger equations \cite{Han2, Han3}. \par 
The modulation spaces $M_{p,q}^{s}(\mathbb{R}^n)$ are built out of a uniform covering $\mathcal{U}(\mathbb{R}^n)$ on $\mathbb{R}^n$, while the Besov spaces $\mathcal{B}_{p,q}^{s}(\mathbb{R}^{n})$ have a dyadic covering $\mathcal{B}(\mathbb{R}^n)$ associated to them. The intermediate spaces $M_{p,q}^{s,\alpha}(\mathbb{R}^n)$ have associated coverings $\mathcal{Q}^{\alpha}(\mathbb{R}^n)$ that interpolate between the extreme cases $\mathcal{U}(\mathbb{R}^n)$ and $\mathcal{B}(\mathbb{R}^n)$. It is advantageous for several of the applications mentioned above to extend the $\alpha$-modulation spaces to a setting that include non-uniform translation and dilation symmetries. Modulation spaces can be defined on locally compact abelian groups \cite{Hans_3}, while the (homogeneous) Besov spaces have been generalized to stratified Lie groups in \cite{Atomic_Besov} through integrability properties of the sub-Laplacian. We aim to extend all the $\alpha$-modulation spaces $M_{p,q}^{s,\alpha}(\mathbb{R}^n)$ to the setting of stratified Lie groups through a more geometric approach that emphasizes the underlying coverings mentioned above. The choice to extend the $\alpha$-modulation spaces to stratified Lie groups is motivated by the desire to obtain the following two properties for the resulting spaces $M_{p,q}^{s,\alpha}(G)$:

\begin{enumerate}[label=(\roman*)]
\label{enumerate_intro}
    \item We can realize all the elements in $M_{p,q}^{s,\alpha}(G)$ as distributions on $\mathbb{R}^n$ where $n = \textrm{dim}(G)$. This will allow us to use the Euclidean Fourier transform in the description of the spaces $M_{p,q}^{s,\alpha}(G)$. 
    \item The fact that any stratified Lie group possesses dilations and a (typically non-abelian) group structure is needed for a satisfying definition of the boundary cases $\alpha = 0,1$.
\end{enumerate}

For a stratified Lie group $G$ it is possible to identify $G$ with $(\mathbb{R}^n,*_G)$ where $n = \textrm{dim}(G)$ and $*_G$ is a product that is polynomial in each component. The initiated reader should have the Heisenberg groups $\mathbb{H}_n$ in mind. The special case $M_{p,q}^{s,0}(\mathbb{H}_n)$ has already been investigated in \cite{David} with the help of representation theory. \par
We are able to define the spaces $M_{p,q}^{s,\alpha}(G)$ for any stratified Lie group $G$. However, we can only assure that the definition is not vacuous whenever the step of $G$ is less than or equal two. The reason for this will be explained and discussed further in Subsection \ref{sec: Almost_Structured_Coverings}. Although we expect the \textit{generalized $\alpha$-modulation spaces} $M_{p,q}^{s,\alpha}(G)$ to be well-defined for all stratified Lie groups $G$, we are not able to show this with current methods. Moreover, for the most part we need to restrict to the stratified Lie groups $G$ being \textit{rational}, meaning that there exists a lattice $N \subset G$. This is a mild condition and is easily checked in practice. Whenever possible, we will state results for arbitrary stratified Lie groups in the hope that the restriction to rational stratified Lie groups with step less than or equal two will be removed in the future. \par 
The two properties \hyperref[enumerate_intro]{(i)} and \hyperref[enumerate_intro]{(ii)} above can be considered as necessary conditions for studying the spaces $M_{p,q}^{s,\alpha}(G)$. However, two generalizations are not equally rewarding and the reader should be skeptical whether this initial outset yields satisfying results. Except for expecting the spaces $M_{p,q}^{s,\alpha}(G)$ to satisfy basic results regarding completeness, duality and so on, the following five questions seem appropriate to answer:

\begin{enumerate}[label=\arabic*)]
\label{main_properties}
    \item Are there coverings $\mathcal{Q}^{\alpha}(G)$ on $\mathbb{R}^n$ associated to the spaces $M_{p,q}^{s,\alpha}(G)$ in the same manner as in the Euclidean setting? Moreover, do these coverings reflect some geometric property of the stratified Lie group $G$ in the uniform case $\alpha = 0$?
    \item Can one use the spaces $M_{p,q}^{s,\alpha}(G)$ for an application without extensive knowledge of stratified Lie groups? That is, can someone with a time-frequency analysis or harmonic analysis background effectively work with the spaces $M_{p,q}^{s,\alpha}(G)$?
    \item Have any of the spaces $M_{p,q}^{s,\alpha}(G)$ previously appeared in the literature? Are the spaces $M_{p,q}^{s,\alpha}(G)$ interesting whenever $G$ is not isomorphic to $(\mathbb{R}^n,+)$ as a Lie group?
    \item Is the extension from $M_{p,q}^{s,\alpha}(\mathbb{R}^n)$ to $M_{p,q}^{s,\alpha}(G)$ uninteresting in the sense that the definitions need only be trivially modified to obtain spaces with analogous properties? Do all the techniques used when studying the Euclidean $\alpha$-modulation spaces $M_{p,q}^{s,\alpha}(\mathbb{R}^n)$ extend in an obvious way to solve the same problems for the spaces $M_{p,q}^{s,\alpha}(G)$?
    \item Are the new spaces $M_{p,q}^{s,\alpha}(G)$ distinct from their Euclidean counterparts $M_{p,q}^{s,\alpha}(\mathbb{R}^n)$? More precisely, is it possible that \[M_{p_1,q_1}^{s_1,\alpha_1}(G) = M_{p_2,q_2}^{s_2,\alpha_2}(\mathbb{R}^n),\] for some parameters $1 \leq p_1, p_2, q_1, q_2 \leq \infty$, $s_1,s_2 \in \mathbb{R}$, and $0 \leq \alpha_1, \alpha_2 \leq 1$?
\end{enumerate}

We will not attempt to address the first four questions in the introduction, but will answer them throughout the paper and return to them again in Section \ref{sec: Looking_Back_and_Ahead}. The fifth question turns out to be the most challenging and the answer given in Theorem \ref{novelty_theorem} can be seen as the main technical achievement of the paper. Our result will extend the known result for the modulation spaces on the Heisenberg group given in \cite[Theorem 7.6]{David}. We say that the parameters $p,q,s,\alpha$ with $1\leq p,q \leq \infty$, $s \in \mathbb{R}$, and $0 \leq \alpha \leq 1$ are \textit{non-trivial} if $(p,q,s) \neq (2,2,0)$. Question \hyperref[main_properties]{5)} above has the following complete answer.

\begin{theorem*}{(Main Theorem)}
\label{introduction_theorem_1}
Let $(\mathbb{R}^n,*_G)$ denote a rational stratified Lie group with step less than or equal two. Consider two sets of non-trivial parameters $1 \leq p_1,p_2,q_1,q_2 \leq \infty$, $s_1,s_2 \in \mathbb{R}$, and $0 \leq \alpha_1,\alpha_2 \leq 1$. We have equality \[M_{p_1,q_1}^{s_1,\alpha_1}(G) = M_{p_2,q_2}^{s_2,\alpha_2}(\mathbb{R}^n)\] with equivalent norms if and only if both \[(p_1,q_1,s_1,\alpha_1) = (p_2,q_2,s_2,\alpha_2) \quad \text{and} \quad (\mathbb{R}^n,*_G) \simeq (\mathbb{R}^n,+).\]
\end{theorem*}

Given two stratified Lie groups $G$ and $H$ with $\textrm{dim}(G) = \textrm{dim}(H)$, the spaces $M_{p_1,q_1}^{s_1,\alpha_1}(G)$ and $M_{p_2,q_2}^{s_2,\alpha_2}(H)$ will both consist of distributions on $\mathbb{R}^n$. Hence it makes sense to ask whether the inclusion $M_{p_1,q_1}^{s_1,\alpha_1}(G) \hookrightarrow M_{p_2,q_2}^{s_2,\alpha_2}(H)$ is bounded for certain parameters. However, when $\textrm{dim}(G) \neq \textrm{dim}(H)$ this approach is not possible as the distributions in each space are not comparable. As a substitute, we would like to understand when there exist embeddings \[F:M_{p_1,q_1}^{s_1,\alpha_1}(G) \longrightarrow M_{p_2,q_2}^{s_2,\alpha_2}(H)\] that preserve the underlying coverings $\mathcal{Q}^{\alpha}(G)$ and $\mathcal{Q}^{\alpha}(H)$ in a suitable sense. These embeddings have recently been invented in \cite{Berge} under the name \textit{geometric embeddings}. We will give the precise definitions in Section \ref{sec: Metric_Geometry_G_alpha_Modulation_Spaces}. The existence of geometric embeddings is at the moment to challenging to answer in its full generality. In Theorem \ref{new_geometric_embedding_result} we give a partial answer to when the Euclidean modulation spaces $M_{p,q}^{s,0}(\mathbb{R}^k)$ can embed geometrically into the generalized modulation spaces $M_{p,q}^{s,0}(G)$. 

\begin{theorem*}
\label{introduction_theorem_2}
Let $G$ be a rational stratified Lie group with step less than or equal two and with rank $k$. There exists a geometric embedding \[F: M_{p,q}^{s,0}(\mathbb{R}^{k'}) \xrightarrow{ } M_{p,q}^{s,0}(G)\] for every $k' \leq k$, $1 \leq p,q < \infty$ and $s \in \mathbb{R}$. This is in general optimal as there are no geometric embeddings from $M_{p,q}^{s,0}(\mathbb{R}^{k'})$ to $M_{p,q}^{s,0}(\mathbb{R}^{l})$ for $l < k'$.
\end{theorem*}

The structure of the paper is as follows: In Section \ref{sec: Preliminaries} we introduce stratified Lie groups, admissible coverings, and related notions. We will also define the traditional $\alpha$-modulation spaces $M_{p,q}^{s,\alpha}(\mathbb{R}^n)$ in Section \ref{sec: alpha_Modulation_Spaces _on _Euclidean _Spaces} to make the exposition more self-contained. The coverings $\mathcal{Q}^{\alpha}(G)$ associated to the group $G$ are defined in Section \ref{sec: G_alpha_coverings} and we develop some of their basic properties. In Subsection \ref{sec: Almost_Structured_Coverings} we discuss when the elements in the covering $\mathcal{Q}^{\alpha}(G)$ are images of a few reference sets under well-behaved affine transformations. As one might expect, this depends on how \textquote{polynomial} the group multiplication on $(\mathbb{R}^{n},*_G)$ is. In Section \ref{sec: G_alpha_modulation_spaces} we define the spaces $M_{p,q}^{s,\alpha}(G)$ and investigate their duality relations. We moreover show that the rapidly decaying smooth functions $\mathcal{S}(\mathbb{R}^n)$ are contained in $M_{p,q}^{s,\alpha}(G)$. It is in Section \ref{sec: Novelty} that we answer the fifth question regarding uniqueness of the spaces $M_{p,q}^{s,\alpha}(G)$ and develop a few auxiliary results. We will study geometric embeddings in Section \ref{sec: Metric_Geometry_G_alpha_Modulation_Spaces}. Finally, in Section \ref{sec: Looking_Back_and_Ahead} we look back on the five questions posted in the introduction and outline some open problems and possible future directions.

\subsection*{Acknowledgements}
The author is grateful for the insightful remarks provided by Hans Feichtinger, Franz Luef, and Felix Voigtlaender. Much of the paper was written when the author visited David Rottensteiner in Vienna, supported by the BFS/TFS project Pure Mathematics in Norway.

\section{Preliminaries}
\label{sec: Preliminaries}

Our notational conventions are fairly standard: We use the convention that $\mathbb{N}$ does not contain zero and we will write $\mathbb{N}_{0} := \mathbb{N} \cup \{0\}$. The Lebesgue measure of a measurable set $A \subset \mathbb{R}^n$ will be denoted by $|A|$, while the number of elements in a finite or countably infinite set $B$ will be denoted by $\#B$. The \textit{Fourier transform} on $\mathbb{R}^{n}$ will be denoted by $\mathcal{F}$ and we use the normalization convention \[\mathcal{F}(f)(\omega) = \int_{\mathbb{R}^n}f(x) \cdot e^{-2\pi i x \cdot \omega} \, dx.\] We will denote by $\mathcal{S}(\mathbb{R}^n)$ the space of smooth functions on $\mathbb{R}^n$ with rapid decay. Its topological dual space $\mathcal{S}'(\mathbb{R}^n)$ will be referred to as the \textit{tempered distributions}. Denote by $L^{p} := L^{p}(\mathbb{R}^n)$ the $p$'th integrable Lebesgue measurable functions for $1 \leq p \leq \infty$ with the usual modification for $p = \infty$. The space $l^{q}(I)$ where $I$ is a countable index set will denote the $q$'th summable sequences indexed by $I$ where $1 \leq q < \infty$. Similarly, the space $l^{\infty}(I)$ denotes all bounded sequences on the index set $I$. When the index set $I$ is clear from the context we will often simply write $l^{q} := l^{q}(I)$ for $1 \leq q \leq \infty$. We will use the notation  $\|\cdot\|_{E}$ for the usual Euclidean norm on $\mathbb{R}^n$ and reserve the notation $\|\cdot\|$ for the homogeneous quasi-norms on stratified Lie groups introduced in Subsection \ref{sec: Stratified_Lie_Groups}.

\subsection{Stratified Lie Groups}
\label{sec: Stratified_Lie_Groups}

In this section we briefly outline the essence of stratified Lie groups and the basic constructions on them we will need in subsequent chapters. As our intended audience include people with a background in harmonic analysis and time-frequency analysis, we have tried to keep the prerequisites at a minimum. Any statement that is not justified in this section can be found in \cite[Chapter 1.6 and 3.1]{Quantization_on_Nilpotent_Lie_Groups}.

\begin{definition}
Let $G$ be a connected and simply connected Lie group with Lie algebra $\mathfrak{g}$. Then $G$ is called \textit{stratified} if there exists a \textit{stratification} 
\begin{equation}
\label{Lie_algebra_decomposition}
    \mathfrak{g} = V_1 \oplus \dots \oplus V_s, \quad [V_1,V_j] = \begin{cases}
V_{j+1}, \, \textrm{ if } \, j = 1, \dots, s-1 \\
\{0\}, \, \, \, \, \textrm{ if } \, j = s
\end{cases}.
\end{equation}
\end{definition}
The number $s$ is called the \textit{step} of $G$ while the number $k := \textrm{dim}(V_1)$ is called the \textit{rank} of $G$. Both numbers are invariant under different choices of stratifications. Elements in $V_i$ are said to be of \textit{degree} $i$ for $i = 1,\dots,s$ and we use the notation $\textrm{deg}(X) = i$ for $X \in V_i$. It is clear that any stratified Lie group $G$ is \textit{nilpotent}, that is, the \textit{adjoint map} $\textrm{ad}_{X}:\mathfrak{g} \to \mathfrak{g}$ given by $\textrm{ad}_{X}(Y) = [X,Y]$ is a nilpotent linear map for all $X \in \mathfrak{g}$. \par
For stratified Lie groups the exponential map $\exp_{G}:\mathfrak{g} \to G$ is a global diffeomorphism and we denote its inverse by $\log_{G}:G \to \mathfrak{g}$. The \textit{Baker-Campbell-Hausdorff formula} (BCH) gives the expression 
\begin{equation}
\label{BCH-formula}
    \log_{G}\left(\exp_{G}(X)*_{G}\exp_{G}(Y)\right) = X + Y + \frac{1}{2}[X,Y] - \frac{1}{12}[Y,[X,Y]]+ \cdots,
\end{equation} where there are only finitely many terms due to the nilpotency and they all involve iterated brackets between $X$ and $Y$. \par
An important feature of stratified Lie groups is that they admit dilations: Define the maps $D_r:\mathfrak{g} \to \mathfrak{g}$ for $r > 0$ by \[D_{r}(X) = r^{\textrm{deg}(X)}X, \quad X \in \mathfrak{g}.\] It is straightforward to see that the maps $D_r$ are all Lie algebra isomorphisms. Since $G$ is the connected and simply connected Lie group of $\mathfrak{g}$, there exist unique Lie group automorphisms $D_r^{G}:G \to G$ lifting the maps $D_r$ for all $r > 0$. We call the maps $D_{r}^{G}:G \to G$ for $r > 0$ \textit{dilations} on the Lie group $G$ and they are explicitly given by \[D_{r}^{G}(g) = \exp_{G} \circ D_r \circ \log_{G}(g), \quad g \in G.\] \par
Any stratified Lie group $G$ is \textit{unimodular}, that is, the right and left Haar measures coincide. Let $\mu$ denote a choice of Haar measure on $G$. Then \begin{equation}
\label{Haar_measure_on_G}
    \mu(A) = \lambda(\log_{G}(A)),
\end{equation}
where $\lambda$ is a corresponding choice of Lebesgue measure on the vector space $\mathfrak{g}$ and $A \subset G$ is a Borel measurable set. Hence $\mu(D_{r}^{G}(A)) = r^{Q}\mu(A)$, where \[Q := \sum_{j = 1}^{s}j \cdot \textrm{dim}(V_j).\] The number $Q$ satisfies $\textrm{dim}(G) \leq Q$ and is called the \textit{homogeneous dimension} of the stratified Lie group $G$. \par
Recall that a \textit{lattice} $N$ in a Lie group $G$ is a discrete subgroup such that there exists a $G$-invariant Borel measure $\mu_{G/N}$ on the quotient $G/N$ with $\mu_{G/N}(G/N) < \infty$. Lattices in stratified Lie groups enjoy two properties that are not shared by lattices in general Lie groups (or in general locally compact groups): 
\begin{itemize}
    \item Any lattice $N$ in a stratified Lie group $G$ is \textit{uniform}, that is, the quotient space $G/N$ is compact. In fact, the compactness of $G/N$ for a discrete subgroup $N$ is equivalent to the existence of a $G$-invariant Borel measure $\mu_{G/N}$ on the quotient $G/N$ with $\mu_{G/N}(G/N) < \infty$ \cite[Theorem 2.1]{Discrete_Subgroups}. 
    \item Any lattice in a stratified Lie group is a finitely generated nilpotent group \cite[Theorem 2.10]{Discrete_Subgroups}. 
\end{itemize}
Moreover, a stratified Lie group $G$ admits a lattice if and only if there exists a basis $X_1, \dots, X_n$ for its Lie algebra $\mathfrak{g}$ such that the \textit{structure constants} $c_{ij}^{k}$ defined by the relation \[[X_i,X_j] = \sum_{k = 1}^{n}c_{ij}^{k}X_k, \quad i,j = 1, \dots, n,\] are all rational numbers \cite[Theorem 2.12]{Discrete_Subgroups}.  Such stratified Lie groups are called \textit{realizable over the rationals} or simply \textit{rational}. The classification of nilpotent Lie algebras in \cite{Lie_Algebras_Low_Dim} shows that every stratified Lie group of dimension less than seven is rational.  We will mostly be interested in stratified Lie groups $G$ that are rational and many results (such as Proposition \ref{explicit_alpha_covering_prop}, Theorem \ref{novelty_theorem}, and Theorem \ref{new_geometric_embedding_result}) require this. \par 
We can identify $G$ as a manifold with $\mathbb{R}^n$ for $n = \textrm{dim}(G)$ through the exponential map. The group operation $*_G$ on $\mathbb{R}^n$ such that $G$ is isomorphic to $(\mathbb{R}^n,*_G)$ as a Lie group is polynomial by the BCH formula \eqref{BCH-formula}. Then relation \eqref{Haar_measure_on_G} shows that the Haar measure $\mu$ on $G$ transported to $\mathbb{R}^n$ through the exponential map is simply the Lebesgue measure $\lambda$ on $\mathbb{R}^n$. However, lattices $N$ in $G$ are not in general identified with the standard lattices in $\mathbb{R}^n$, that is, the subgroups $\Gamma \subset \mathbb{R}^n$ on the form $\Gamma = A\mathbb{Z}^n$, where $A \in GL(n,\mathbb{R})$. Our motivation for identifying stratified Lie groups with $\mathbb{R}^n$ comes from the need to use the Euclidean Fourier transform when we define the generalized $\alpha$-modulation spaces $M_{p,q}^{s,\alpha}(G)$ in Section \ref{sec: G_alpha_modulation_spaces}. \par 
Let us now describe an alternative to the usual Euclidean norm $\|\cdot\|_{E}$ on $\mathbb{R}^n$ that is adapted to the stratified Lie group $G$: We say that a function $f:G \to \mathbb{C}$ is $l$-\textit{homogeneous} for $l \in \mathbb{N}_0$ if \[f\left(D_{r}^{G}(g)\right) = r^{l}f(g),\] for every $r > 0$ and all $g \in G$. The function $f:G \to \mathbb{C}$ is called \textit{symmetric} if $f(g) = f\left(g^{-1}\right)$ for every $g \in G$.
\begin{definition}
A \textit{homogeneous quasi-norm} on a stratified Lie group $G$ is a 1-homogeneous continuous function that is symmetric and has the property that $\|g\| = 0$ only holds when $g$ is the identity element of $G$.
\end{definition}
We will use the standard notation \[B^{\|\cdot\|}(g,R) := \left\{h \in G \, \Big| \, \|g^{-1} *_{G} h\| < R \right\}, \qquad g \in G, \quad R > 0.\] The following proposition is proved in \cite[Proposition 3.1.35]{Quantization_on_Nilpotent_Lie_Groups} and shows that the choice of homogeneous quasi-norm is in many instances irrelevant. 
\begin{lemma}
\label{homogeneous_quasi_norm_result}
Let $G$ be a stratified Lie group. Then $G$ admits a homogeneous quasi-norm that is smooth away from the identity element. Moreover, any two homogeneous quasi-norms $\|\cdot\|_{1}$ and $\|\cdot\|_{2}$ on $G$ are equivalent in the sense that there exists $C > 0$ such that \[\frac{1}{C}\|g\|_{1} \leq \|g\|_{2} \leq C \|g\|_{1},\] for every $g \in G$.
\end{lemma}
The terminology \textquote{quasi-norm} is justified by \cite[Proposition 3.1.38]{Quantization_on_Nilpotent_Lie_Groups}, showing that homogeneous quasi-norms satisfy 
\begin{equation}
\label{quasi_norm_constant}
    \|g *_{G} h\| \leq C\left(\|g\| + \|h\|\right), \quad g,h \in G,
\end{equation}
where $C \geq 1$ is a constant that does not depend on the elements $g,h \in G$. In fact, it is always possible by \cite[Proposition 3.1.39]{Quantization_on_Nilpotent_Lie_Groups} to find a \textit{homogeneous norm}, that is, a homogeneous quasi-norm $\|\cdot\|$ that additionally satisfies \[\|g *_{G} h\| \leq \|g\| + \|h\|, \quad g,h \in G.\] \par
When considering stratified Lie groups in the rest of this paper, we implicitly assume the following standing assumption: We always chose the realization of $G$ as $(\mathbb{R}^n,*_G)$ where $n = \textrm{dim}(G)$ through the exponential map. The triple $(\mathbb{R}^n, *_G, \|\cdot\|)$ will for the rest of the paper denote the realization of $G$ where $\|\cdot\|$ is a choice of a homogeneous quasi-norm on $(\mathbb{R}^n, *_{G})$.

\begin{example}
\label{Heisenberg_example}
Consider the \textit{Heisenberg Lie algebra} $\mathfrak{h}_{n} = \textrm{span}_{\mathbb{R}}\{X_1, \dots, X_n, Y_1, \dots, Y_n, Z\}$ with non-trivial bracket relations \[[X_i,Y_i] = Z, \quad i = 1, \dots, n.\] The connected and simply connected Lie group $\mathbb{H}_{n}$ corresponding to $\mathfrak{h}_{n}$ is called the \textit{Heisenberg group}. It follows from the BCH formula \eqref{BCH-formula} that \[\log_{\mathbb{H}_{n}}\left(\exp_{\mathbb{H}_{n}}(X) *_{\mathbb{H}_{n}} \exp_{\mathbb{H}_{n}}(Y)\right) = X + Y + \frac{1}{2}[X,Y], \quad X,Y \in \mathfrak{h}_n.\] Through the exponential map, the Heisenberg group $\mathbb{H}_{n}$ is isomorphic as a Lie group to $(\mathbb{R}^{2n+1},*_{\mathbb{H}_n})$ where \[(x,\omega,t) *_{\mathbb{H}_{n}} (x',\omega',t) := \left(x + x', \, \omega + \omega', \, t + t' + \frac{1}{2}(x'\omega - x \omega')\right),\] for $x,x',\omega,\omega' \in \mathbb{R}^n$ and $t,t' \in \mathbb{R}$. After this identification, the dilations $D_{r}^{\mathbb{H}_n}$ for $r > 0$ are given by \[D_{r}^{\mathbb{H}_{n}}(x,\omega,t) = (rx,r\omega,r^{2}t), \quad (x,\omega,t) \in \mathbb{R}^{2n+1}.\] The homogeneous dimension of $\mathbb{H}_{n}$ is $Q = 2n + 2$ and a concrete example of a lattice in $(\mathbb{R}^{2n+1},*_{\mathbb{H}_n})$ is \[N := \left\{(x,\omega,t) \in \mathbb{R}^{2n+1} \, \Big| \, x,\omega \in 2\mathbb{Z}^n, \, t \in \mathbb{Z}\right\}.\] Moreover, the \textit{homogeneous Cygan-Koranyi norm}  \begin{equation}
\label{homogeneous_Cygan-Koranyi_norm}
    (x,\omega,t) \longmapsto \left((|x|^2 + |\omega|^2)^2 + 16t^2\right)^{\frac{1}{4}}
\end{equation} is an example of a homogeneous quasi-norm on the Heisenberg group. 
\end{example}

\subsection{Admissible Coverings}
\label{sec: Admissible_coverings}

We give a brief review of admissible coverings and some related notions that we need in subsequent sections. Admissible coverings was originally formulated in \cite{Hans_Grobner} as special coverings on an arbitrary set. However, we will restrict ourselves to admissible coverings on $\mathbb{R}^n$ since every stratified Lie group $G$ has a realization as $(\mathbb{R}^n,*_{G})$ as explained in the previous section.

\begin{definition}
A covering $\mathcal{Q} = (Q_i)_{i \in I}$ consisting of non-empty sets on $\mathbb{R}^n$ is called \textit{admissible} if we have the uniform bound \begin{equation}
\label{admissibility_constant}
    \sup_{i \in I}\#\left\{j \in I \, \Big| \, Q_i \cap Q_j \neq \emptyset \right\} \leq N_{\mathcal{Q}}, 
\end{equation}
for some $N_{\mathcal{Q}} \in \mathbb{N}$.
The admissible covering $\mathcal{Q}$ will be called a \textit{concatenation} if we additionally have the equality \begin{equation}
\label{concatenation_property}
    \mathbb{R}^n = \bigcup_{k = 1}^{\infty}Q_{i}^{k*},
\end{equation}
for some (and hence all) $i \in I$, where we use the notation \[Q_{i}^{*}:= \left\{Q_j \in \mathcal{Q} \, \Big| \, Q_i \cap Q_j \neq \emptyset \right\}, \qquad Q_{i}^{k*} := \left(Q_{i}^{(k-1)*}\right)^{*},\] for $k \geq 2$ and $i \in I$.
\end{definition}

Given an admissible covering $\mathcal{Q} = (Q_i)_{i \in I}$ we call the elements in $Q_i^{*}$ the \textit{neighbours} of the set $Q_i \in \mathcal{Q}$. Moreover, the smallest possible constant $N_{\mathcal{Q}}$ in \eqref{admissibility_constant} is called the \textit{admissibility constant} of the admissible covering $\mathcal{Q}$. The admissibility condition \eqref{admissibility_constant} is needed to obtain non-trivial classes of functions that have a prescribed frequency decay with respect to the covering $\mathcal{Q}$. On the other hand, the concatenation property \eqref{concatenation_property} will be necessary when we examine coverings from a metric space viewpoint in Section \ref{sec: Novelty} and Section \ref{sec: Metric_Geometry_G_alpha_Modulation_Spaces}. \par 
We will in Section \ref{sec: Novelty} need the notion of weight functions that are well-behaved with respect to an admissible covering $\mathcal{Q} = (Q_i)_{i \in I}$ on $\mathbb{R}^n$. To be precise, we will call a function $\omega:I \to (0,\infty)$ $\mathcal{Q}$\textit{-moderate} if we have the uniform bound \[\sup_{\{j: Q_i \cap Q_j \neq \emptyset \}}\frac{\omega(i)}{\omega(j)} \leq \mathcal{C}_{\omega},\] where the constant $\mathcal{C}_{\omega}$ does not depend on the index $i \in I$.\par

Given two admissible coverings $\mathcal{Q} = (Q_i)_{i \in I}$ and $\mathcal{P} = (P_{j})_{j \in J}$ on $\mathbb{R}^n$, there are two common ways of comparing them:

\begin{itemize}
    \item We say that $\mathcal{Q}$ is \textit{almost subordinate} to $\mathcal{P}$ if there exists a $k \in \mathbb{N}$ such that for every $i \in I$ there is a $j \in J$ with $Q_{i} \subset P_{j}^{k*}$. We use the notation $\mathcal{Q} \leq \mathcal{P}$ and say that the coverings $\mathcal{Q}$ and $\mathcal{P}$ are \textit{equivalent} if both $\mathcal{Q} \leq \mathcal{P}$ and $\mathcal{P} \leq \mathcal{Q}$ are satisfied.
    \item We say that $\mathcal{Q}$ is \textit{weakly subordinate} to $\mathcal{P}$ if we have the bound \[\sup_{i \in I}\#\left\{j \in J \, \Big| \, P_j \cap Q_i \neq \emptyset\right\} < \infty.\] If $\mathcal{Q}$ is weakly subordinate to $\mathcal{P}$ and vice versa, we call the coverings \textit{weakly equivalent}.
\end{itemize}

It follows from \cite[Proposition 3.5]{Hans_Grobner} that almost subordination implies weak subordination, although the converse is not true in general. It is generally difficult to show that one covering $\mathcal{Q}$ is almost subordinate to another covering $\mathcal{P}$. However, it is often easier to show that $\mathcal{Q}$ is weakly subordinate to $\mathcal{P}$. Whenever the coverings consist of open and path-connected sets, then it follows from \cite[Proposition 3.6]{Hans_Grobner} that the two notions coincide.

An arbitrary admissible covering $\mathcal{Q} = (Q_{i})_{i \in I}$ on $\mathbb{R}^{n}$ can have (at least) two annoying features: Firstly, the index set $I$ might not be countable. Secondly, subsets $Q_i \in \mathcal{Q}$ are allowed to be repeated several times in the collection $\mathcal{Q} = (Q_i)_{i \in I}$, only with different indices. The covering $\mathcal{Q}$ on $\mathbb{R}$ whose index set is $I = \mathbb{R} \times \{0,1\}$ and $Q_{(r,0)} = Q_{(r,1)} = \{r\}$ for $r \in \mathbb{R}$ is a simple admissible covering that embodies both problems simultaneously. Moreover, the covering $\mathcal{Q}$ is clearly not a concatenation as it is a partition. The following lemma shows that these problems disappear once we require the elements in the covering to be open sets.

\begin{lemma}
\label{lemma_about_basic_properties_of_coverings}
Let $\mathcal{Q} = (Q_i)_{i \in I}$ be an admissible covering on $\mathbb{R}^n$ consisting of open sets. Then $I$ has to be countable and the covering $\mathcal{Q}$ is automatically a concatenation. Moreover, we can always remove the repeated elements in $\mathcal{Q} = (Q_i)_{i \in I}$ and obtain an equivalent covering. 
\end{lemma}

\begin{proof}
Since $\mathcal{Q}$ is an open covering on $\mathbb{R}^n$ we can find a countable subcovering $\mathcal{Q}' = (Q_j)_{j \in J}$ of $\mathcal{Q}$ with $J \subset I$. Consider the sets \[A_{j} := \left\{i \in I \, \Big| \, Q_i \cap Q_j \neq \emptyset \right\}, \quad j \in J.\] Then for every $i \in I$ we can find a set $A_j$ with $j \in J$ such that $i \in A_{j}$ since $\mathcal{Q}'$ is a covering on $\mathbb{R}^n$. The set $\cup_{j \in J}A_{j}$ is countable and we obtain that $I$ has to be countable as well. \par 
The concatenation property \eqref{concatenation_property} is equivalent to the following statement: Given $x,y \in \mathbb{R}^n$ we can find a sequence $Q_{i_1}, \dots, Q_{i_k} \in \mathcal{Q}$ of elements in $\mathcal{Q}$ with $x \in Q_{i_1}$ and $y \in Q_{i_k}$  such that $Q_{i_l} \cap Q_{i_{l+1}} \ne \emptyset$ for every $1 \leq l \leq k-1$. Such a sequence is called a \textit{chain} from $x$ to $y$ in \cite{Hans_Grobner}. To see that this is always possible to find, consider the straight line \begin{equation}
\label{straight_line}
    \gamma_{x,y}:[0,1] \to \mathbb{R}^{n}, \qquad \gamma_{x,y}(t) = ty + (1-t)x,
\end{equation} connecting $x$ and $y$. Since the image of $\gamma_{x,y}$ is compact and the elements in $\mathcal{Q}$ are open, we can find a finite set of elements $(Q_j)_{j \in J}$ in $\mathcal{Q}$ such that \[\textrm{Im}(\gamma_{x,y}) \subset \bigcup_{j \in J}Q_j.\] A standard topological argument using the openness of the elements $(Q_j)_{j \in J}$ shows that we can reorder $(Q_j)_{j \in J}$ to obtain a chain from $x$ to $y$. The final statement is obvious from the definition of equivalent coverings.
\end{proof}

\begin{remark}
We would like to emphasize that the proof of Lemma \ref{lemma_about_basic_properties_of_coverings} goes through if, instead of $\mathbb{R}^n$, we consider a path-connected topological space $X$ where any open covering on $X$ has a countable subcover. The only modification is that we would need to pick an abstract continuous path from $x$ to $y$ guaranteed by the path-connectedness of $X$ rather than the straight line given in \eqref{straight_line}. These conditions hold for all connected manifolds and hence include most settings considered in the literature. 
\end{remark}

\subsection{$\alpha$-Modulation Spaces}
\label{sec: alpha_Modulation_Spaces _on _Euclidean _Spaces}

We now give the definitions of the Euclidean $\alpha$-coverings and $\alpha$-modulation spaces $M_{p,q}^{s,\alpha}(\mathbb{R}^n)$. This will serve as a motivation for the generalization to stratified Lie groups described in the next sections.

\begin{definition}
\label{definition_Euclidean_alpha_covering}
An admissible covering $\mathcal{Q}^{\alpha} = (Q_{i}^{\alpha})_{i \in I}$ on $\mathbb{R}^{n}$ consisting of open and connected sets is called an $\alpha$-\textit{covering} for $0 \leq \alpha \leq 1$ if 
\begin{itemize}
    \item The sets $Q_{i}^{\alpha} \in \mathcal{Q}^{\alpha}$ satisfy $|Q_{i}^{\alpha}| \asymp (1 + \|\xi_i\|_{E}^2)^{\frac{\alpha n}{2}}$ for all $\xi_i \in Q_{i}^{\alpha}$.
    \item For each $i \in I$ we denote by $r\left(Q_{i}^{\alpha}\right)$ and $R\left(Q_{i}^{\alpha}\right)$ the numbers  
\begin{align*}
    r\left(Q_{i}^{\alpha}\right) & := \sup \left\{r \in \mathbb{R} \, \Big| \, B(c_r,r) \subset Q_{i}^{\alpha} \textrm{ for some } c_r \in \mathbb{R} \right\}, \\ 
    R\left(Q_{i}^{\alpha}\right) & := \inf \left\{R \in \mathbb{R} \, \Big| \, Q_{i}^{\alpha} \subset B(C_r,R)\textrm{ for some } C_r \in \mathbb{R} \right\}.
\end{align*}
There should exists a constant $K \geq 1$ such that 
\begin{equation}
\label{euclidean_ratio_bound}
    \sup_{i \in I} \frac{R\left(Q_{i}^{\alpha}\right)}{r\left(Q_{i}^{\alpha}\right)} \leq K.
\end{equation}
\end{itemize}
\end{definition}

There is much variation in the literature about the definition of $\alpha$-coverings: In \cite{BN2} the authors do not require \eqref{euclidean_ratio_bound} to hold. In \cite[Chapter 9]{Felix_main}, the author considers a concrete covering that satisfies Definition \ref{definition_Euclidean_alpha_covering}. Our definition is the same as in \cite{BN1} and is motivated by the following remark.

\begin{remark}
\label{weakly_equivalent_remark}
It follows from \cite[Lemma B.2]{BN1} that any two $\alpha$-coverings on $\mathbb{R}^n$ as we have defined them are weakly equivalent. Thus they are in fact equivalent since they consist of open and connected sets. To see that connectedness is a necessary condition, we can take $\mathcal{Q} = (Q_n)_{n \in \mathbb{Z}}$ to be the covering $Q_{n} = (n - 1, n + 1)$ and $\mathcal{P} = (U_k)_{k \in \mathbb{N}_0}$ to be the covering \[U_0 = (-2,2), \quad U_{k} = (-k-2,-k) \cup (k,k+2), \quad k \in \mathbb{N}.\] Both coverings are $0$-coverings on $\mathbb{R}$. However, they are clearly not equivalent since $\mathcal{P}$ is not almost subordinate to $\mathcal{Q}$. 
\end{remark}

Let $\mathcal{Q} = (Q_i)_{i \in I}$ be an admissible covering on $\mathbb{R}^n$. A \textit{(smooth) bounded admissible partition of unity} subordinate to $\mathcal{Q}$ ($\mathcal{Q}$-BAPU) is a family of non-negative smooth functions $\Phi = (\psi_{i})_{i \in I}$ on $\mathbb{R}^{n}$ such that 
\begin{equation}
\label{BAPU_definition}
    \textrm{supp}(\psi_{i}) \subset Q_i, \qquad \sum_{i \in I}\psi_{i} \equiv 1, \qquad \sup_{i \in I}\left\|\mathcal{F}^{-1}\psi_{i}\right\|_{L^1} < \infty.
\end{equation}

\begin{definition}
\label{usual_alpha_mod_spaces}
Let $\mathcal{Q}^{\alpha} = (Q_{i}^{\alpha})_{i \in I}$ be an $\alpha$-covering on $\mathbb{R}^n$ and let $\Phi = (\psi_{i})_{i \in I}$ be a $\mathcal{Q}^{\alpha}$-BAPU. For $1 \leq p,q \leq \infty$, $s \in \mathbb{R}$ and $0 \leq \alpha \leq 1$ we define the $\alpha$\textit{-modulation space} $M_{p,q}^{s,\alpha}(\mathbb{R}^n)$ to be all tempered distributions $f \in \mathcal{S}'(\mathbb{R}^n)$ such that \[\|f\|_{M_{p,q}^{s,\alpha}} := \left(\sum_{i \in I}\left(1 +  \|\xi_{i}\|_{E}^{2}\right)^{\frac{qs}{2}}\left\|\mathcal{F}^{-1}\left(\psi_{i} \cdot \mathcal{F}(f)\right)\right\|_{L^{p}}^{q}\right)^{\frac{1}{q}} < \infty,\] where $\xi_{i} \in Q_{i}^{\alpha}$ for every $i \in I$. If $q = \infty$ we use the obvious modification from summation to supremum.
\end{definition}

The $\alpha$-modulation spaces $M_{p,q}^{s,\alpha}(\mathbb{R}^n)$ were first introduced in \cite{GrobnerThesis}. Since two $\alpha$-coverings on $\mathbb{R}^{n}$ are equivalent, we obtain from \cite[Theorem 3.7]{Hans_Grobner} that the resulting $\alpha$-modulation space does not depend on the choice of $\alpha$-covering. Moreover, they do not depend on the choice of bounded admissible partition of unity either by \cite[Theorem 2.3 B)]{Hans_Grobner}. 

\begin{example}
\label{usual_modulation_spaces_definition}
If $\alpha = 0$ then an option for an $\alpha$-covering on $\mathbb{R}^n$ is the \textit{uniform covering} \[\mathcal{U}(\mathbb{R}^n) = (Q_{m_1,\dots,m_n})_{m_1,\dots,m_n \in \mathbb{Z}}, \qquad Q_{m_1,\dots,m_n} := (-1,1)^{n} + (m_1,\dots,m_n).\] The resulting spaces $M_{p,q}^{s}(\mathbb{R}^n) := M_{p,q}^{s,0}(\mathbb{R}^n)$ are precisely the \textit{modulation spaces with polynomial weights}. They are typically denoted by $M_{v_s}^{p,q}(\mathbb{R}^n)$ in the literature and we refer the reader to \cite[Chapter 11]{TF_analysis} for more information on them.
\end{example}

\begin{example}
\label{dyadic_covering_example}
If $\alpha = 1$ and $n \geq 2$ we can use the \textit{dyadic covering} $\mathcal{B}(\mathbb{R}^n) = (D_{m})_{m = 0}^{\infty}$ where $D_{0} = B(0,2)$ and 
\begin{equation}
\label{Besov_Covering}
    D_{m} = \left\{x \in \mathbb{R}^n \, \Big | \, 2^{m-1} < \|x\|_{E} < 2^{m+1}\right\}, \quad m \in \mathbb{N}.
\end{equation}
The resulting spaces $\mathcal{B}_{p,q}^{s}(\mathbb{R}^n) := M_{p,q}^{s,0}(\mathbb{R}^n)$ are the \textit{(inhomogeneous) Besov spaces}. For $n = 1$ the covering given in \eqref{Besov_Covering} is not connected and we would need to split each of the sets $D_{m}$ for $m \geq 1$ into its two connected components and consider them individually to obtain a $1$-covering.
\end{example}

To summarize, the $\alpha$-modulation spaces are a one-parameter class of Banach spaces connecting the modulation spaces used in time-frequency analysis and the Besov spaces used in harmonic analysis. 

\section{Generalized $\alpha$-Coverings}
\label{sec: G_alpha_coverings}

We will in this section define generalized $\alpha$-coverings for $0 \leq \alpha \leq 1$ on $\mathbb{R}^n$ that reflect the stratified Lie group structure $(\mathbb{R}^n,*_G)$ and extend the $\alpha$-coverings defined in Subsection \ref{sec: alpha_Modulation_Spaces _on _Euclidean _Spaces}.  We will emphasize the role of the homogeneous quasi-norms and lattices in the stratified Lie group $(\mathbb{R}^n,*_G)$.

\subsection{Definition and Equivalence}
\label{sec:Definition_and_Equivalence}

From now on, we identify a stratified Lie group $G$ with $(\mathbb{R}^n,*_G)$ through the exponential map and fix a homogeneous quasi-norm $\| \cdot \|$ on $(\mathbb{R}^n,*_G)$. By doing this, we have to keep track of that $\mathbb{R}^n$ is equipped with a both a group structure $*_G$ and a Lie algebra structure $\mathbb{R}^n \simeq \mathfrak{g} = V_1 \oplus \dots \oplus V_s$, where $s$ is the step of $G$. When writing elements $x \in (\mathbb{R}^n,*_G)$ in coordinates $x = (x_1, \dots, x_n)$ we implicitly assume that we have chosen a basis $v_1,\dots, v_n$ for $\mathbb{R}^n$ that is \textit{adapted to the stratification}. This means that $v_1, \dots, v_{\textrm{dim}(V_1)}$ is a basis for $V_1,$ $v_{\textrm{dim}(V_1) + 1}, \dots, v_{\textrm{dim}(V_2)}$ is a basis for $V_2$, and so on. 

\begin{definition}
\label{definition_alpha_coverings}
Let $(\mathbb{R}^n,*_G,\|\cdot\|)$ be a stratified Lie group with homogeneous dimension $Q$ where $\|\cdot\|$ is a chosen homogeneous quasi-norm. For a fixed $0 \leq \alpha \leq 1$ we call an admissible covering $\mathcal{P}^{\alpha} = (P_{i}^{\alpha})_{i \in I}$ on $\mathbb{R}^n$ consisting of open and connected sets a \textit{generalized $\alpha$-covering} for the stratified Lie group $(\mathbb{R}^n,*_G)$ if it satisfies the following two properties:

\begin{itemize}
    \item The sets $P_{i}^{\alpha}$ satisfy the estimates
    \begin{equation}
\label{size_of_coverings}
    |P_{i}^{\alpha}| \asymp \left(1+\|\xi_{i}\|^2\right)^{\frac{\alpha Q}{2}},
\end{equation}
for all $\xi_i \in P_{i}^{\alpha}$ and every $i \in I$.
\item For each $i \in I$ we denote by $r^{\|\cdot\|}\left(P_{i}^{\alpha}\right)$ and $R^{\|\cdot\|}\left(P_{i}^{\alpha}\right)$ the numbers  
\begin{align*}
    r^{\|\cdot\|}\left(P_{i}^{\alpha}\right) & = \sup \left\{r \in \mathbb{R} \, \Big| \, B^{\|\cdot\|}(c_r,r) \subset P_{i}^{\alpha} \textrm{ for some } c_r \in \mathbb{R} \right\}, \\ 
    R^{\|\cdot\|}\left(P_{i}^{\alpha}\right) & = \inf \left\{R \in \mathbb{R} \, \Big| \, P_{i}^{\alpha} \subset B^{\|\cdot\|}(C_r,R)\textrm{ for some } C_r \in \mathbb{R} \right\}.
\end{align*}
There should exists a constant $K \geq 1$ such that 
\begin{equation}
\label{ratio_condition}
    \sup_{i \in I} \frac{R^{\|\cdot\|}\left(P_{i}^{\alpha}\right)}{r^{\|\cdot\|}\left(P_{i}^{\alpha}\right)} \leq K.
\end{equation}
\end{itemize}
\end{definition}

Notice that the numbers $r^{\|\cdot\|}\left(P_{i}^{\alpha}\right)$ and $R^{\|\cdot\|}\left(P_{i}^{\alpha}\right)$ are strictly positive since we assume that the sets $P_{i}^{\alpha}$ are open. Condition \eqref{ratio_condition} is necessary to obtain that two generalized $\alpha$-coverings are equivalent as will be shown in Proposition \ref{proposition_about_equivalence}. Notice that Lemma \ref{lemma_about_basic_properties_of_coverings} implies that the index set $I$ has to be countable and that the covering $\mathcal{P}^{\alpha}$ is automatically a concatenation.

\begin{remark}
When $G = (\mathbb{R}^n,+)$ the homogeneous dimension satisfies $Q = n$ and we regain the definition of the $\alpha$-coverings given in Definition \ref{definition_Euclidean_alpha_covering}. The reason for realizing generalized $\alpha$-coverings corresponding to stratified Lie groups on Euclidean space is to involve the Euclidean Fourier transform when we define generalized $\alpha$-modulation spaces in Section \ref{sec: G_alpha_modulation_spaces}. Moreover, the heuristic reason we use the homogeneous dimension $Q$ in \eqref{size_of_coverings} instead of the dimension $n$ is that we would like to obtain \textquote{Besov type spaces} for $\alpha = 1$ that incorporate the intrinsic dilations $D_{r}^{G}$ of the stratified Lie group $(\mathbb{R}^n,*_G)$. We will see later in Lemma \ref{Besov_covering_lemma} that this intuition gives a concrete $1$-covering that is similar to the dyadic covering given in Example \ref{dyadic_covering_example}. \par 
\end{remark}

Whenever $\alpha = 0$ then condition \eqref{size_of_coverings} simply says that the Lebesgue measure of the sets $P_{i}^{0}$ is constant. One could wonder whether the uniform covering $\mathcal{U}(\mathbb{R}^n)$ in Example \ref{usual_modulation_spaces_definition} satisfies \eqref{ratio_condition} and thus is a generalized $\alpha$-covering for any rational stratified Lie group $(\mathbb{R}^n,*_G)$ other than $(\mathbb{R}^n,+)$. We will see in Proposition \ref{Besov_are_all_the_same} that this is not the case by using arguments from metric geometry. The reader might get some motivation for this approach by trying to prove this statement directly without additional tools. \par 

\begin{example}
For the Heisenberg group $\mathbb{H}_3$ with the homogeneous Cygan-Koranyi norm \eqref{homogeneous_Cygan-Koranyi_norm} we have that generalized $\alpha$-coverings $\mathcal{P}^{\alpha} = (P_{i}^{\alpha})_{i \in I}$ for $0 \leq \alpha \leq 1$ satisfy \[|P_{i}^{\alpha}| \asymp \left(1 + \left((x^2 + \omega^2)^2 + 16t^2\right)^{\frac{1}{2}}\right)^{2\alpha} \asymp \left(1 + x^4 + \omega^4 + 2x^2\omega^2 + 16t^2\right)^{\alpha},\] for $(x,\omega,t) \in P_{i}^{\alpha}$.
\end{example}

We give explicit examples of generalized $\alpha$-coverings in Subsection \ref{sec: Concrete_Examples}. Before that, we turn to the question about equivalence. The following proposition implies that the specific choice of homogeneous quasi-norm and generalized $\alpha$-covering does not matter when we define the generalized $\alpha$-modulation spaces $M_{p,q}^{s,\alpha}(G)$ in Section \ref{sec: G_alpha_modulation_spaces}. The proof of the second statement in Proposition \ref{proposition_about_equivalence} is inspired by the proof of the corresponding statement for Euclidean $\alpha$-coverings given in \cite[Appendix B]{BN1}.

\begin{proposition}
\label{proposition_about_equivalence}
A covering $\mathcal{P}^{\alpha}$ on $\mathbb{R}^n$ is a generalized $\alpha$-covering for the stratified Lie group $(\mathbb{R}^n,*_G)$ independently of the choice of the homogeneous quasi-norm. Moreover, any two generalized $\alpha$-coverings $\mathcal{Q}^{\alpha}$ and $\mathcal{P}^{\alpha}$ for $(\mathbb{R}^n,*_G)$ are equivalent.
\end{proposition}

\begin{proof}
Assume that $\| \cdot \|_1$ and $\| \cdot \|_2$ are two homogeneous quasi-norms on $(\mathbb{R}^n,*_G)$ and that $\mathcal{P} = (P_{i}^{\alpha})_{i \in I}$ is a covering that satisfies \eqref{size_of_coverings} with respect to $\|\cdot\|_{1}$. It follows from Lemma \ref{homogeneous_quasi_norm_result} that \[|P_{i}^{\alpha}| \asymp (1+\|\xi_{i}\|_{1}^2)^{\frac{\alpha Q}{2}} \asymp (1+\|\xi_{i}\|_{2}^2)^{\frac{\alpha Q}{2}},\] for all $\xi_i \in P_{i}^{\alpha}$. Similarly, we have $r^{\|\cdot\|_{1}}(P_{i}^{\alpha}) \asymp r^{\|\cdot\|_{2}}(P_{i}^{\alpha})$ and $R^{\|\cdot\|_{1}}(P_{i}^{\alpha}) \asymp R^{\|\cdot\|_{2}}(P_{i}^{\alpha})$ independently of $i \in I$. Hence condition \eqref{ratio_condition} is satisfied for $\|\cdot\|_2$ when it is satisfied for $\|\cdot\|_{1}$ and the first statement follows. \par 
For the last statement, it suffices by \cite[Proposition 3.6]{Hans_Grobner} to show that $\mathcal{Q}^{\alpha}$ and $\mathcal{P}^{\alpha}$ are weakly equivalent since they both consist of connected and open sets. Let us fix a homogeneous quasi-norm $\| \cdot \|$ on $(\mathbb{R}^n,*_{G})$ and denote by $\mu := |B^{\|\cdot\|}(0,1)|$. We first claim that 
\begin{equation}
\label{comparable_sizes}
    |P_{i}^{\alpha}| \asymp \left(r^{\|\cdot\|}(P_{i}^{\alpha})\right)^{Q} \asymp \left(R^{\|\cdot\|}(P_{i}^{\alpha})\right)^{Q},
\end{equation}
where $Q$ denotes the homogeneous dimension of $(\mathbb{R}^n,*_G)$. Since the usual Lebesgue measure is the Haar measure on $(\mathbb{R}^n,*_{G})$ we have \[|B^{\|\cdot\|}(x,R)| = |x *_{G} B^{\|\cdot\|}(0,R)| = |B^{\|\cdot\|}(0,R)| = |D_{R}^{G}B^{\|\cdot\|}(0,1)| = R^{Q}\mu,\] where $x \in \mathbb{R}^n$ is arbitrary and $R > 0$. Hence \[\mu \cdot \left(r^{\|\cdot\|}(P_{i}^{\alpha})\right)^{Q} \leq |P_{i}^{\alpha}| \leq \mu \cdot \left(R^{\|\cdot\|}(P_{i}^{\alpha})\right)^{Q},\] for every $i \in I$. This implies \eqref{comparable_sizes} since \[\mu \leq \frac{|P_{i}^{\alpha}|}{\left(r^{\|\cdot\|}(P_{i}^{\alpha})\right)^{Q}} = \frac{\left(R^{\|\cdot\|}(P_{i}^{\alpha})\right)^{Q}}{\left(r^{\|\cdot\|}(P_{i}^{\alpha})\right)^{Q}} \cdot \frac{|P_{i}^{\alpha}|}{\left(R^{\|\cdot\|}(P_{i}^{\alpha})\right)^{Q}} \leq K^{Q}\mu,\] where $K \geq 1$ denotes the uniform bound in \eqref{ratio_condition}. \par 
Assume that $Q_{j}^{\alpha} \cap P_{i}^{\alpha}  \neq \emptyset$ for some $i \in I$ and some $j \in J$. Then \eqref{size_of_coverings} together with \eqref{comparable_sizes} give the estimate \[R^{\|\cdot\|}(Q_{j}^{\alpha}) \asymp R^{\|\cdot\|}(P_{i}^{\alpha}) \asymp r^{\|\cdot\|}(P_{i}^{\alpha}).\] Hence there exists a uniform constant $\kappa \geq 1$ such that 
\begin{equation}
\label{inside_ball}
    Q_{j}^{\alpha} \subset B^{\|\cdot\|}\left(c_i,\kappa r^{\|\cdot\|}(P_{i}^{\alpha})\right), \quad c_i \in P_{i}^{\alpha}.
\end{equation}
For every $i \in I$ we consider the constants \[A_{\alpha}(i) := \# \left\{Q_{j}^{\alpha} \in \mathcal{Q}^{\alpha} \, \Big| \, Q_{j}^{\alpha} \cap P_{i}^{\alpha}  \neq \emptyset \right\}.\]
Assume that there exist a sequence $i_k \in I$ with $k \in \mathbb{N}$ such that $A_{\alpha}(i_k) \to \infty$. Then if $Q_{j_k}^{\alpha} \cap P_{i_k}^{\alpha} \neq \emptyset$, we have
\begin{equation}
\label{last_equation_in_proof}
    \frac{|Q_{j_k}^{\alpha}|}{\Big|B^{\|\cdot\|}\left(c_{i_k},\kappa r^{\|\cdot\|}(P_{i_k}^{\alpha})\right)\Big|} \asymp \left(\frac{1}{\kappa \mu}\right)^{Q}.
\end{equation} 
Notice that the right hand side of \eqref{last_equation_in_proof} does not depend on $k \in \mathbb{N}$. Thus \eqref{inside_ball} and \eqref{last_equation_in_proof} give a contradiction since $\mathcal{Q}^{\alpha}$ is assumed to be admissible.  
\end{proof}

Notice that Proposition \ref{proposition_about_equivalence} still leaves open the possibility that a generalized $\alpha_1$-covering $\mathcal{P}^{\alpha_1}$ and a generalized $\alpha_2$-covering $\mathcal{Q}^{\alpha_2}$ for a stratified Lie group $(\mathbb{R}^n,*_G)$ might be equivalent whenever $\alpha_1 \neq \alpha_2$. We will prove in Theorem \ref{weakly_subordinate_result} that is not possible.

\subsection{Concrete Examples}
\label{sec: Concrete_Examples}

We now turn to giving concrete examples of generalized $\alpha$-coverings. It will be clear that we need to require that $(\mathbb{R}^n,*_G)$ is rational in the intermediate case $0 < \alpha < 1$. For $\alpha = 0$, the existence of a lattice $N \subset (\mathbb{R}^n,*_G)$ is convenient but not nessesary. For $\alpha = 1$ the existence of a lattice is irrelevant. The main difference from the Euclidean case is that we do not have the luxury of picking the \textquote{canonical} lattice $\mathbb{Z}^n$.

\subsubsection{The Uniform Case: $\alpha = 0$}
\label{sec: Uniform_Case}

We would like to find a concrete generalized $0$-covering for $(\mathbb{R}^n, *_G)$ that, similarly to the Euclidean case, reflects the group operation $*_G$. Such a covering was constructed for any locally compact group in \cite{Hans_2} and we briefly review this in the setting of stratified Lie groups realized on Euclidean space. \par
For every stratified Lie group $(\mathbb{R}^n,*_G, \|\cdot\|)$ there exists a covering $\mathcal{U}$ on $\mathbb{R}^n$ constructed in the following manner: Fix the set $B^{\|\cdot\|}(0,1)$ and consider the collection \[\left \{x *_{G} B^{\|\cdot\|}(0,1)\right\}_{x \in \mathbb{R}^n} = \left \{B^{\|\cdot\|}(x,1)\right\}_{x \in \mathbb{R}^n}.\]

\begin{lemma}
\label{uniform_lemma}
There exists a family of elements $\{x_i\}_{i \in I}$ with $x_i \in \mathbb{R}^n$ for every $i \in I$ such that \[\mathcal{U}(G) := \left \{B^{\|\cdot\|}(x_i,1)\right\}_{i \in I}\] is a generalized $0$-covering for the stratified Lie group $(\mathbb{R}^n,*_G)$. 
\end{lemma}

\begin{proof}
It follows from \cite{Hans_2} that there exists a family of elements $\{x_i\}_{i \in I}$ with $x_i \in \mathbb{R}^n$ for every $i \in I$ such that $\mathcal{U}(G)$ is an admissible covering. To show that any ball \[B^{\|\cdot\|}(x,R), \quad x \in \mathbb{R}^n, \, R > 0\] is path-connected, it suffices to consider the unit ball $B_{0} := B^{\|\cdot\|}(0,1)$ by applying a left-translation and a scaling $D_{R}^{G}$. The path $t \mapsto (tx_1, \dots,t^{v_j}x_j, \dots, t^{v_n}x_n)$ for $t \in [0,1]$, $v_j := \textrm{deg}(x_j)$,  and $x := (x_1, \dots,x_n) \in B_{0}$ connects the origin to $x$ and lies within $B_0$ since \[\|(tx_1, \dots, t^{v_n}x_n)\| = \|D_{t}^{G}(x_1, \dots, x_n)\| = |t|\|x\| < 1.\] The balls $B^{\|\cdot\|}(x,R)$ are also open due to the continuity of the homogeneous quasi-norm $\|\cdot\|$. \par 
We are left with checking the two conditions in the definition of a generalized $\alpha$-covering: The first condition \eqref{size_of_coverings} follows readily since \[\left|B^{\|\cdot\|}(x_i,R)\right| = R^Q \left|B^{\|\cdot\|}(0,1)\right| \simeq \left(1 + \|\xi_i\|^{2}\right)^{\frac{0 \cdot Q}{2}},\] where $\xi_i \in B^{\|\cdot\|}(x_i,R)$ and $Q$ is the homogeneous dimension of $(\mathbb{R}^n,*_G)$. The second condition \eqref{ratio_condition} is clearly satisfied with $K = 1$ since the covering consists of balls with respect to the homogeneous quasi-norm $\|\cdot\|$.
\end{proof}

Proposition \ref{proposition_about_equivalence} implies that the choice of the family $\{x_i\}_{i \in I}$ is largely irrelevant as different families will produce equivalent coverings. We refer to $\mathcal{U}(G)$ as the \textit{uniform covering} of the stratified Lie group $(\mathbb{R}^n,*_G)$. \par
In the case where the stratified Lie group $(\mathbb{R}^{n}, *_G)$ is rational, we can be even more concrete: Fix a lattice $N \subset (\mathbb{R}^{n}, *_G)$ and fix $R > 0$ such that the collection
\begin{equation}
\label{spesific_uniform_covering}
    \mathcal{U}(G;N) := \left\{B^{\|\cdot\|}(n,R)\right\}_{n \in N \setminus \{0\}}
\end{equation}
is an admissible covering. This is possible since $N$ is both uniform and discrete. We can again apply Proposition \ref{proposition_about_equivalence} to see that $\mathcal{U}(G;N)$ is equivalent to the uniform covering $\mathcal{U}(G)$ and we consider $\mathcal{U}(G;N)$ as a concrete realization of $\mathcal{U}(G)$. When $G = (\mathbb{R}^n,+)$ and $N = \mathbb{Z}^n$ then the covering $\mathcal{U}(G;N)$ is precisely the covering introduced in Example \ref{usual_modulation_spaces_definition}. Hence the uniform covering $\mathcal{U}(G)$ is a $0$-covering that incorporates information about the group structure $*_G$. 

\subsubsection{The Intermediate Case: $0 < \alpha < 1$}
\label{sec: Intermediate_Case}

We turn to the intermediate range $0 < \alpha < 1$ and give a concrete covering motivated by the most commonly used $\alpha$-covering in the Euclidean setting in e.g. \cite[Chapter 9]{Felix_main}. This covering will require the existence of a lattice $N \subset (\mathbb{R}^n,*_G)$ and extends the covering $\mathcal{U}(G;N)$ introduced above.  \par 

\begin{proposition}
\label{explicit_alpha_covering_prop}
Let $(\mathbb{R}^n,*_G,\|\cdot\|)$ be a rational stratified Lie group with a lattice $N \subset (\mathbb{R}^n, *_G)$. We will use the notation \[\delta_{\beta}(\xi) := \|\xi\|^{\beta} \xi, \quad \xi \in \mathbb{R}^n, \qquad \beta := \frac{\alpha}{1-\alpha},\] where we have fixed $0 \leq \alpha < 1$. There exists $r_1 > 0$ such that the collection \begin{equation}
\label{explicit_alpha_covering}
    \mathcal{Q}_{r}^{\alpha}(G;N) := \left\{B^{\|\cdot\|}\left(\delta_{\beta}(k),r\|k\|^\beta\right)\right\}_{k \in N\setminus \{0\}}
\end{equation} is a generalized $\alpha$-covering for any $r > r_1$. For $\alpha = 0$ the covering $\mathcal{Q}_{r}^{0}(G;N)$ is simply $\mathcal{U}(G;N)$ introduced previously.
\end{proposition}

\begin{proof}
Since the statement about $\alpha = 0$ is clear and already justified previously, we will henceforth assume that $0 < \alpha < 1$. The topology induced on $\mathbb{R}^n$ by the balls with respect to the homogeneous quasi-norm $\|\cdot\|$ is equivalent to the usual Euclidean topology by \cite[Proposition 3.1.37]{Quantization_on_Nilpotent_Lie_Groups}. Hence since $N$ is uniform we have that there exists $r_1 > 0$ such that the covering $\mathcal{Q}_{r}^{\alpha}(G;N)$ is a covering for all $r > r_1$. The argument that $\mathcal{Q}_{r}^{\alpha}(G;N)$ is admissible is the same as in the Euclidean case and is given in \cite[Lemma 2.5 and Theorem 2.6]{BN1}. We can duplicate the proof of Lemma \ref{uniform_lemma} to deduce all the properties needed for $\mathcal{Q}_{r}^{\alpha}(G;N)$ to be a generalized $\alpha$-covering except for the proof of condition \eqref{size_of_coverings}. To show this, we need to make a few estimates: \par
If we let $Q$ denote the homogeneous dimension of $(\mathbb{R}^n,*_G)$, then we have
\begin{equation}
\label{first_equation_spesific_covering}
    \Big|B^{\|\cdot\|}(\delta_{\beta}(k),r\|k\|^\beta)\Big|  =  \Big|B^{\|\cdot\|}(0,r\|k\|^\beta)\Big|  = \Big|D_{r\|k\|^\beta}^{G}B^{\|\cdot\|}(0,1)\Big| \asymp (r\|k\|^\beta)^{Q} \asymp  \|k\|^{\frac{\alpha Q}{1-\alpha}},
\end{equation}
by the left-invariance and dilation properties of the Haar measure. By picking the center point $\delta_{\beta}(k)$ in the ball $B^{\|\cdot\|}(\delta_{\beta}(k),r\|k\|^\beta)$ we see that
\begin{equation}
\label{second_equation_spesific_covering}
    \left(1 + \|\delta_{\beta}(k)\|^2\right)^{\frac{\alpha Q}{2}} = \left(1 + \|k\|^{2\beta + 2}\right)^{\frac{\alpha Q}{2}} = \left(1 + \|k\|^{\frac{2}{1-\alpha}}\right)^{\frac{\alpha Q}{2}}.
\end{equation}
Since we have excluded zero from the lattice $N$ the estimate $\|k\|^{\frac{2}{1-\alpha}} \asymp 1 + \|k\|^{\frac{2}{1-\alpha}}$ is valid. Comparing this observation with \eqref{first_equation_spesific_covering} and \eqref{second_equation_spesific_covering} shows that the covering $\mathcal{Q}_{r}^{\alpha}(G;N)$ is a generalized $\alpha$-covering since \[\Big|B^{\|\cdot\|}(\delta_{\beta}(k),r\|k\|^\beta)\Big|^{\frac{2}{\alpha Q}} \asymp \|k\|^{\frac{2}{1-\alpha}} \asymp 1 + \|k\|^{\frac{2}{1-\alpha}} = 1 + \|\delta_{\beta}(k)\|^2.\] \par 
In the definition of a generalized $\alpha$-covering, we need the above estimate for every $\xi_k \in B^{\|\cdot\|}(\delta_{\beta}(k),r\|k\|^\beta)$ and not only the center point $\delta_{\beta}(k)$. This follows from a straightforward computation using that \[\|\delta_{\beta}(k)\| = \|\delta_{\beta}(k) - \xi_k + \xi_k\| \leq C(r\|k\|^{\beta} + \|\xi_k\|),\] where $C > 0$ is the constant appearing in \eqref{quasi_norm_constant}. 
\end{proof}
 
\subsubsection{The Dyadic Case: $\alpha = 1$}
\label{sec: Dyadic_Case}

The covering given in Proposition \ref{explicit_alpha_covering_prop} is clearly not well-defined for $\alpha = 1$. We will give a concrete example of a generalized $1$-covering that models the classical dyadic intervals underlying the Besov spaces given in Example \ref{dyadic_covering_example}.

\begin{definition}
\label{definition_Besov_covering}
Let $(\mathbb{R}^n,*_G,\|\cdot\|)$ be a stratified Lie group. The covering $\mathcal{B}(G) = \{D_{m}(G)\}_{m \in \mathbb{N}_0}$ given by \[D_{0}(G) = B^{\|\cdot\|}(0,2), \quad D_{m}(G) = B^{\|\cdot\|}(0,2^{m+1}) \setminus \overline{B^{\|\cdot\|}(0,2^{m-1})}, \quad m \in \mathbb{N},\] 
is called the \textit{Besov covering} with respect to the homogeneous quasi-norm $\|\cdot\|$ on the stratified Lie group $(\mathbb{R}^n,*_G)$.
\end{definition}

The fact that the homogeneous quasi-norm $\|\cdot\|$ is not part of the notation $\mathcal{B}(G)$ will be justified in Lemma \ref{Besov_covering_lemma}. The Besov covering is an admissible covering consisting of open and connected sets. Hence it is a concatenation by Lemma \ref{lemma_about_basic_properties_of_coverings}. The most important property of the covering $\mathcal{B}(G)$ is the scaling invariance \[D_{2^k}^{G}D_{m}(G) = D_{m+k}(G), \quad m \geq 1, \, k \geq 0.\] \par
For an arbitrary homogeneous quasi-norm $\|\cdot\|$, one can not assure that $(2^m, 0,\dots, 0) \in D_{m}(G)$ for $m \geq 1$. Although this is not a serious obstacle, we can fix this by using the homogeneous quasi-norm
\begin{equation}
\label{spesific_homogeneous_quasi_norm}
    \|(x_1, \dots,x_n)\|_{2} := \left(\sum_{j = 1}^{n}|x_j|^{\frac{2}{v_j}}\right)^{\frac{1}{2}},
\end{equation}
where $v_j$ denotes the degree of $x_j$. We emphasize that we denote the usual Euclidean norm by $\|\cdot\|_E$ to distinguish it from the homogeneous quasi-norm $\|\cdot\|_{2}$ in \eqref{spesific_homogeneous_quasi_norm}. If $(\mathbb{R}^n,*_G)$ has rank $k$ then \[\|(x_1, \dots, x_k ,0, \dots, 0)\|_{2} = \|(x_1, \dots, x_k ,0, \dots, 0)\|_{E}.\] With the homogeneous quasi-norm \eqref{spesific_homogeneous_quasi_norm} it is clear that \begin{equation}
\label{elemetns_in_dyadic_covering}
    (2^{m},0,\dots,0) \in D_{m}(G), \quad m \in \mathbb{N}_0.
\end{equation}
Moreover, the group structure between the elements in \eqref{elemetns_in_dyadic_covering} is the same as the Euclidean addition since they are in the first layer $V_1$. The following lemma shows that fixing the homogeneous quasi-norm $\|\cdot\|_{2}$ is justified and that the Besov covering is a concrete realization of a generalized $1$-covering.

\begin{lemma}
\label{Besov_covering_lemma}
Assume that $n > 1$. The Besov covering $\mathcal{B}(G)$ is a $1$-covering for the stratified Lie group $(\mathbb{R}^n,*_G)$ independently of the choice of homogeneous quasi-norm.
\end{lemma}

\begin{proof}
We first work with the homogeneous quasi-norm $\|\cdot\|_{2}$ given in \eqref{spesific_homogeneous_quasi_norm}. Let us begin by checking that the ratio property \eqref{ratio_condition} is satisfied: For $m = 0$ we obviously have $r^{\|\cdot\|_2}(D_{0}(G)) = R^{\|\cdot\|_2}(D_{0}(G))$ and hence the ratio is one. For $m \geq 1$ we claim that we have the estimates \[R^{\|\cdot\|_{2}}(D_{m}(G)) \leq 2^{m+1}, \qquad r^{\|\cdot\|_{2}}(D_{m}(G)) \geq 2^{m-1}.\] The first is clear from the definition of $R^{\|\cdot\|_{2}}(D_{m}(G))$ while the second follows from considering a ball centered at the point $c_m = (2^{m}, \dots, 0)$. In conclusion, this gives 
\[\sup_{m\in \mathbb{N}} \frac{R^{\|\cdot\|_{2}}\left(D_{m}(G)\right)}{r^{\|\cdot\|_{2}}\left(D_{m}(G)\right)} \leq \sup_{m\in \mathbb{N}} \frac{2^{m+1}}{2^{m-1}} = 4.\] \par
To see that the size condition \eqref{size_of_coverings} is satisfied, we denote by $\mu := |B^{\|\cdot\|_{2}}(0,1)|$ and estimate for $m \geq 1$ that 
\begin{align*}
    |D_{m}(G)| & = \left|B^{\|\cdot\|_{2}}(0,2^{m+1})\right| - \left|B^{\|\cdot\|_{2}}(0,2^{m-1})\right| \\
               & = \left|D_{2^{m+1}}^{G}B^{\|\cdot\|_{2}}(0,1)\right| - \left|D_{2^{m-1}}^{G}B^{\|\cdot\|_{2}}(0,1)\right| \\ 
               & = \left(\mu \frac{4^{Q} - 1}{2^Q}\right)2^{Qm}.
\end{align*}
On the other hand, for any $\xi_m \in D_{m}(G)$ we have $\|\xi_m\| \asymp 2^{m}$ and hence \[(1 + \|\xi_m\|_{2}^{2})^{\frac{Q}{2}} \asymp 2^{Qm}.\] 
Combining these estimates shows that the Besov covering $\mathcal{B}(G)$ with the homogeneous quasi-norm $\|\cdot\|_{2}$ is a $1$-covering. Then we can apply Proposition \ref{proposition_about_equivalence} and obtain that the choice of homogeneous quasi-norm defining the Besov covering $\mathcal{B}(G)$ is irrelevant as they all produce equivalent coverings. Hence we can safely use the homogeneous quasi-norm $\|\cdot\|_2$ given in \eqref{spesific_homogeneous_quasi_norm} without loss of generality. When $n = 1$ we have that $(\mathbb{R},*_G) \simeq (\mathbb{R},+)$. In that case, we refer the reader to Example \ref{dyadic_covering_example} for a trivial modification of the result.
\end{proof}

Let $(\mathbb{R}^n,*_{G})$ be a rational stratified Lie group with a lattice $N$.  We will use the notation
\begin{equation}
\label{explicit_coverings_cases}
    \mathcal{Q}^{\alpha}(G) := \mathcal{Q}^{\alpha}(G;N) = \begin{cases} 
\mathcal{U}(G;N), & \textrm{if } \alpha = 0 \\ 
\mathcal{Q}_{r}^{\alpha}(G;N), & \textrm{if } 0 < \alpha < 1 \\ 
\mathcal{B}(G), & \textrm{if } \alpha = 1
\end{cases},
\end{equation}
where the number $r > 0$ is chosen large enough so that $\mathcal{Q}_{r}^{\alpha}(G;N)$ is a concatenation. The specific value of $r > 0$ needed will be suppressed as it is of no relevance in our augments. 

\begin{remark}
We have showed that the concatenation $\mathcal{Q}^{\alpha}(G)$ depends (up to equivalence of coverings) only on the parameter $0 \leq \alpha \leq 1$ and the stratified Lie group $(\mathbb{R}^n,*_G)$. We can say even more by introducing the following terminology: The \textit{growth vector} of a stratified Lie group $(\mathbb{R}^n,*_G)$ is the multi-index \[\mathfrak{G}(G) := (n_{1}, \dots, n_{s}), \qquad n_{i} := \textrm{dim}(V_i), \quad i = 1, \dots , s,\] where $V_i$ are as in \eqref{Lie_algebra_decomposition}. If $(\mathbb{R}^n,*_G)$ and $(\mathbb{R}^n,*_H)$ are two stratified Lie groups with $\mathfrak{G}(G) = \mathfrak{G}(H)$ then the homogeneous quasi-norm $\|\cdot\|_{2}$ given in \eqref{spesific_homogeneous_quasi_norm} are equal for both $(\mathbb{R}^n,*_G)$ and $(\mathbb{R}^n,*_H)$. Moreover, they clearly have the same homogeneous dimension as well. Hence by using the homogeneous quasi-norm $\|\cdot\|_2$ we see from Definition \ref{definition_alpha_coverings} that a generalized $\alpha$-covering for $(\mathbb{R}^n,*_G)$ is also a generalized $\alpha$-covering for $(\mathbb{R}^n,*_H)$ and vice versa. From this we can conclude from Proposition \ref{proposition_about_equivalence} that any generalized $\alpha$-covering $\mathcal{P}^{\alpha}$ can be described by two parameters: The continuous parameter $0 \leq \alpha \leq 1$ and the discrete parameter $\mathfrak{G}(G) \in \mathbb{N}_{0}^{s}$ where $s$ is the step of $(\mathbb{R}^n,*_G)$. 
\end{remark}

\subsection{Almost Structured Coverings and BAPU's}
\label{sec: Almost_Structured_Coverings}

Many coverings that arise in practice have the property that its elements are essentially given by well-behaved affine transformations of a few reference sets. This notion was studied in \cite{BN3} and the following definition is a slight generalization appearing in \cite{Felix_main}.

\begin{definition}
\label{almost_structured_definition}
Let $\mathcal{Q} = (Q_i)_{i \in I}$ be an admissible covering on $\mathbb{R}^n$. We call $\mathcal{Q}$ an \textit{almost structured covering} if there exists a finite collection $(\mathcal{P}_{s})_{s \in J}$ of bounded, open subsets of $\mathbb{R}^n$ called \textit{reference sets} with the following properties: 
\begin{itemize}
    \item There is an invertible affine transformation $A_i = T_i + b_i$ for every $i \in I$ with $T_i \in GL(n,\mathbb{R})$ and $b_i \in \mathbb{R}^n$ such that 
    \begin{equation}
    \label{affine_transformation}
        Q_i = A_{i}(P_s) = T_{i}(P_s) + b_i,
    \end{equation} for some $s \in J$ depending on $i \in I$.
    \item If $Q_i \cap Q_j \neq \emptyset$ for some $i,j \in I$ then we have the uniform compatibility condition 
    \begin{equation}
    \label{compatibility_estimate_matricies}
        \|T_{i}^{-1}T_{j}\| \leq \mathcal{C}_{Q} < \infty,
    \end{equation}
    where $\mathcal{C}_{Q}$ does not depend on $i,j \in I$.
    \item There should exists a finite collection $(P_{s}^{'})_{s \in J}$ of open sets with $\overline{P_{s}^{'}} \subset P_{s}$ for every $s \in J$ such that $(A_{i}(P_{s}^{'}))_{i \in I, s \in J}$ cover $\mathbb{R}^n$.
\end{itemize}
If the index set $J = \{s\}$ is a singleton, then the covering $\mathcal{Q}$ is called a \textit{structured covering}.
\end{definition}

\begin{remark}
\begin{itemize}
    \item The elements in a almost structured covering $\mathcal{Q} = (Q_i)_{i \in I}$ are automatically open by \eqref{affine_transformation}. Hence the index set $I$ is always countable and $\mathcal{Q}$ is a concatenation by Lemma \ref{lemma_about_basic_properties_of_coverings}. The reason one needs to consider almost structured coverings rather than structured coverings can be seen from the dyadic covering $\mathcal{B}(\mathbb{R}^2)$ given in Example \ref{dyadic_covering_example}. 
    \item We would also like to point out that structured or almost structured coverings are not  preserved under equivalence of coverings: It is straightforward to construct an equivalent covering to, say, the uniform covering $\mathcal{U}(\mathbb{R}^2)$ that is not even almost structured. Hence questions such as \textquote{are all generalized $\alpha$-coverings corresponding to a stratified Lie group $(\mathbb{R}^n,*_G)$ almost structured?} are not well-defined. To ask meaningful questions, we will have to consider the specific representative coverings given in \eqref{explicit_coverings_cases}.  
\end{itemize}
\end{remark}

The following lemma is proved in \cite[Lemma 9.3]{Felix_main} and shows that the standard realization of the Euclidean $\alpha$-coverings considered in the literature are indeed almost structured coverings.

\begin{lemma}
The coverings $\mathcal{Q}^{\alpha}(\mathbb{R}^{n};\mathbb{Z}^n)$ are structured for $0 \leq \alpha < 1$, while the coverings $\mathcal{B}(\mathbb{R}^n)$ for $n \geq 2$ are only almost structured.
\end{lemma}

Hence one might expect that the coverings $\mathcal{Q}^{\alpha}(G)$ are all at least almost structured for any rational stratified Lie group $(\mathbb{R}^n,*_G)$. This is supported by the fact that a small modification of \cite[Proposition 6.2]{David} shows that the uniform covering $\mathcal{U}(\mathbb{H}_n; N)$ is a structured covering where $\mathbb{H}_n$ is the Heisenberg group and $N$ is the lattice $N = (2\mathbb{Z})^{2n} \times \mathbb{Z}$. However, the following proposition shows that this is not true in general and depends on the step of the stratified Lie group in question.

\begin{proposition}
\label{almost_structured_coverings_result}
Let $(\mathbb{R}^n,*_G)$ be a stratified Lie group where $n > 1$. \begin{itemize}
    \item The Besov covering $\mathcal{B}(G)$ is an almost structured covering that is never structured unless the group $(\mathbb{R}^n,*_G)$ is isomorphic to $(\mathbb{R},+)$.
    \item Assume that $(\mathbb{R}^n,*_G)$ is rational and let $N$ be a lattice. The coverings $\mathcal{Q}^{\alpha}(G;N)$ are structured for $0 \leq \alpha < 1$ whenever the step of $(\mathbb{R}^n,*_G)$ is less than or equal two. However, the coverings $\mathcal{Q}^{\alpha}(G;N)$ for $0 \leq \alpha < 1$ are not necessarily almost structured whenever the step of $(\mathbb{R}^n,*_G)$ is higher than two.
\end{itemize}
\end{proposition}

\begin{proof}
For the Besov covering $\mathcal{B}(G) = (D_{m}(G))_{m = 0}^{\infty}$ given in Definition \ref{definition_Besov_covering} we consider $D_{0}(G)$ and $D_{1}(G)$ as the reference sets. Define the matrices \[A_{m} = T_{m} := \begin{pmatrix}
    2^{(m-1) \cdot v_1} & & \\
    & \ddots & \\
    & & 2^{(m-1) \cdot v_n}
\end{pmatrix}, \qquad v_j := \textrm{deg}(x_j), \quad m \geq 1.\]
By setting $A_0$ to be the identity matrix we then have that $A_{0}(D_{0}(G)) = D_{0}(G)$ and $A_{m}(D_{1}(G)) = D_{m}(G)$ since \[A_{m}D_{1}(G) = D_{2^{m-1}}^{G}D_{1}(G) = D_{m}(G).\]
Notice that two elements $D_{m}(G)$ and $D_{l}(G)$ can only intersect when $m \in \{l-1,l,l+1\}$. In any case, a straightforward computation gives that \[\|T_{l}^{-1}T_{m}\|  \leq 2^{v_n},\] and the estimate \eqref{compatibility_estimate_matricies} is satisfied. Finally, the last requirement in Definition \ref{almost_structured_definition} is clearly satisfied by shrinking $D_{0}(G)$ and $D_{1}(G)$ slightly. The Besov covering $\mathcal{B}(G)$ is not a structured covering when $n > 1$ since $D_{0}(G)$ is convex while the sets $D_{m}(G)$ for $m \geq 1$ are not. When $n = 1$ the stratification \eqref{Lie_algebra_decomposition} has only one layer and hence $(\mathbb{R},*_G) \simeq (\mathbb{R},+)$. It is clear from the construction given in Example \ref{Besov_Covering} that the modification of the covering $\mathcal{B}(\mathbb{R})$ given by dividing each of the sets $D_{m}(\mathbb{R})$ for $m \in \mathbb{N}$ into its connected components is structured. \par 
Let us now turn to the second statement. If $(\mathbb{R}^n, *_G)$ has step one, then we are in the Euclidean setting and the result follows from Proposition \ref{almost_structured_coverings_result}. Assume that $(\mathbb{R}^n, *_G, \|\cdot\|)$ has step two and write $\mathbb{R}^n = \mathbb{R}^{k} \oplus \mathbb{R}^l$ with $l = n - k$ according to the decomposition given in \eqref{Lie_algebra_decomposition}. For $(a,b), (c,d) \in \mathbb{R}^{k} \oplus \mathbb{R}^l$ we can use the BCH-formula \eqref{BCH-formula} to write their product as 
\[(a,b) *_{G} (c,d) = \left(a + c, b + d + \frac{1}{2}P(a,c)\right),\] where $P(a,c)$ is a linear polynomial in the components of $a$ and $c$. This can be written as the block-matrix equation \[(a,b) *_{G} (c,d) = \begin{pmatrix}I_{k \times k} & 0_{l \times k} \\ \rho(c) & I_{l \times l}\end{pmatrix} \begin{pmatrix} a \\ b \end{pmatrix} + \begin{pmatrix} c \\ d \end{pmatrix},\]
where $\rho(c) \in M_{l \times k}(\mathbb{R})$, each of the entries in $\rho(c)$ depend linearly on the components of $c$, and $\rho(c) \cdot a = P(a,c)$. Consider now the covering $\mathcal{Q}^{\alpha}(G;N)$ for $0 \leq \alpha < 1$ and write each element as \[B^{\|\cdot\|}\left(\delta_{\beta}(k),r\|k\|^\beta\right) = \|k\|^{\frac{\alpha}{1-\alpha}}k *_{G}\left( D_{r\|k\|^{\frac{\alpha}{1-\alpha}}}^{G}B^{\|\cdot\|}(0,1)\right), \quad \delta_{\beta}(k) := \|k\|^{\beta}k, \quad \beta := \frac{\alpha}{1 - \alpha},\]
for every $k \in N \setminus \{0\}$. We set the reference set to be $B^{\|\cdot\|}(0,1)$ and leave it to the reader to show that the affine transformations \[A_{k}(x) = \|k\|^{\frac{\alpha}{1-\alpha}}k *_{G} D_{r\|k\|^{\frac{\alpha}{1-\alpha}}}^{G}(x) ,\qquad x \in \mathbb{R}^n, \quad k \in N \setminus \{0\},\] make $\mathcal{Q}^{\alpha}(G;N)$ into a structured covering. \par
Since $\mathcal{Q}^{\alpha}(G;N)$ is almost structured whenever $\mathcal{U}(G;N)$ is almost structured, it suffices find a stratified Lie group $(\mathbb{R}^n,*_G)$ such that $\mathcal{U}(G;N)$ is not almost structured. Consider the stratified Lie group $G$ whose Lie algebra $\mathfrak{g}$ is given by \[\mathfrak{g} = \textrm{span}\{X_1,X_2,X_3,X_4\},\] with bracket relations $[X_1,X_2] = X_3$ and $[X_1,X_3] = X_4$. This is a stratification where $V_1 = \textrm{span}\{X_1,X_2\}$, $V_2 = \textrm{span}\{X_3\}$ and $V_3 = \textrm{span}\{X_4\}.$ By using the BCH-formula \eqref{BCH-formula} we can identify $G$ with $(\mathbb{R}^4,*_G)$, where the multiplication between $(x_1,x_2,x_3,x_4)$ and $(y_1,y_2,y_3,y_4)$ has the form 
\begin{equation}
\label{Engel_multiplication}
   \left(x_1 + y_1, x_2 + y_2, x_3 + y_3 + \frac{1}{2}\left(x_1 y_2 - x_2 y_1\right), x_4 + y_4 + \frac{1}{2}\left(x_1 y_3 - x_3 y_1\right) + \frac{x_1}{12}\left(x_1 y_2 - x_2 y_1\right)\right).
\end{equation}
Consider the lattice $N = 12\mathbb{Z} \times 2 \mathbb{Z} \times \mathbb{Z} \times \mathbb{Z}$ in $(\mathbb{R}^4, *_G)$. Assume first that $\mathcal{Q}(G;N)$ is a structured covering and let $B$ be the reference set. Then for $k,k' \in N \setminus \{0\}$ we can find affine transformations $A_k$ and $A_{k'}$ such that \[A_{k}(B) = T_{k}(B) + b_k = B^{\|\cdot\|}(k,R), \quad A_{k'}(B) = T_{k'}(B) + b_{k'} = B^{\|\cdot\|}(k',R),\] where $R > 0$ is a fixed number so that $\mathcal{Q}(G;N)$ is an admissible covering. Then \[T_{k'}T_{k}^{-1}B^{\|\cdot\|}(k,R) + (b_{k'} - T_{k'}T_{k}^{-1}b_{k}) = B^{\|\cdot\|}(k',R).\] Hence if $k = (12,2,1,1)$ and $k' = (12n,2n,n,n)$ for $n \in \mathbb{N}$ we can increase $n$ and obtain a contradiction due to the quadratic term in $x_1$ in the last entry of \eqref{Engel_multiplication}. This argument can easily be extended to show that $\mathcal{Q}(G;N)$ is not an almost structured covering since infinitely many of the numbers $k' = (12n,2n,n,n)$ for $n \in \mathbb{N}$ have to correspond to one of the (finite number of) reference sets.
\end{proof}

The fact that the coverings $\mathcal{Q}^{\alpha}(G;N)$ are almost structured whenever the step of $(\mathbb{R}^n,*_G)$ is less than or equal two is closely related to the existence of $\mathcal{Q}^{\alpha}(G;N)$-BAPU's. The following proposition follows from \cite[Theorem 2.8]{Felix_Sobolev_BV} which is a slight generalization of the the general existence result \cite[Proposition 1]{BN3}.

\begin{proposition}
Let $(\mathbb{R}^n,*_G)$ be a rational stratified Lie group with step less than or equal two and fix a lattice $N$. Then the exists a $\mathcal{Q}^{\alpha}(G;N)$-BAPU for all $0 \leq \alpha \leq 1$. 
\end{proposition}

\begin{example}
\label{BAPU_example}
Consider a rational stratified Lie group $(\mathbb{R}^n,*_G, \|\cdot\|)$ of step less than or equal two with a lattice $N$. 
\begin{itemize}
    \item For $0 \leq \alpha < 1$, then an explicit $\mathcal{Q}^{\alpha}(G;N)$-BAPU can be constructed by adapting the argument in \cite[Proposition A.1]{BN1} as follows: Fix $r > r_1$ where $r_1$ is the number appearing in Proposition \ref{explicit_alpha_covering_prop}. Consider a positive and smooth function $\Phi:\mathbb{R}^n \to \mathbb{R}$ such that 
\begin{equation}
\label{infimum_and_support_conditions}
    \textrm{supp}(\Phi) \subset B^{\|\cdot\|}(0,r), \quad \inf_{\xi \in B^{\|\cdot\|}(0,r_1)}\Phi(\xi) > 0.
\end{equation}
To find the $\mathcal{Q}_{r}^{\alpha}(G;N)$-BAPU we need, we simply scale the argument of $\Phi$ correctly: Define \[g_{k}(\xi) := \Phi\left(D_{\|c_k\|^{-\alpha}}^{G}\left(c_{k}^{-1}*_{G}\xi\right)\right), \qquad c_k := \|k\|^{\frac{\alpha}{1-\alpha}}k, \quad k \in N\setminus \{0\}.\] Then $g_k$ is smooth and unwinding its definition shows that \[\textrm{supp}(g_k) \subset B^{\|\cdot\|}\left(\delta_{\beta}(k),r\|k\|^\beta\right), \quad k \in N\setminus \{0\}.\]
Moreover, the infimum bound in \eqref{infimum_and_support_conditions} ensures that for every $\xi \in \mathbb{R}^n$ there is a $g_k$ such that $g_{k}(\xi) > 0$. Define \[\psi_k(\xi) := \frac{g_{k}(\xi)}{\sum_{l \in N \setminus \{0\}}g_{l}(\xi)}.\] The $L^{1}$-bound in \eqref{BAPU_definition} is satisfied by an adaption of the argument in \cite[Proposition 2.4]{BN2}. For this to work it is essential that the step of $(\mathbb{R}^n,*_G)$ is less than or equal two so that the group multiplication $*_G$ can be represented by linear maps. Hence we obtain a $\mathcal{Q}_{r}^{\alpha}(G;N)$-BAPU. The functions $\psi_k$ have compact support since the balls induced by the homogeneous quasi-norm $\|\cdot\|$ are bounded (not uniformly) with respect to the Euclidean metric. 
\item For $\alpha = 1$ we can proceed as follows: Pick a positive and smooth function $\Phi_0$ with $\textrm{supp}(\Phi_0) \subset D_{0}(G)$ and $\Phi_{0}(x) = 1$ for every $x \in \mathbb{R}^n$ with $\|x\| \leq \frac{3}{2}$. Moreover, pick a positive and smooth function $\Phi_1$ with $\textrm{supp}(\Phi_1) \subset D_{1}(G)$ and with $\Phi_{1}(x) = 1$ for every $x \in \mathbb{R}^n$ with $\frac{3}{2} \leq \|x\| \leq \frac{7}{2}$. The collection $(\Phi_{m})_{m = 0}^{\infty}$ given by \begin{equation}
\label{Besov_BAPU}
    \Phi_{m}(x) := \Phi_{1}\left(D_{2^{1-m}}^{G}(x)\right), \quad m \geq 2, \, x \in \mathbb{R}^n,
\end{equation}
consists of smooth functions with $\textrm{supp}(\Phi_{m}) \subset D_{m}(G)$ that are never vanishing simultaneously. Hence we define the normalized collection \[\psi_{m}(x) := \frac{\Phi_{m}(x)}{\sum_{k = 0}^{\infty} \Phi_{k}(x)} = \frac{\Phi_{m}(x)}{\Phi_{m-1}(x) + \Phi_{m}(x) + \Phi_{m+1}(x)}, \qquad x \in \mathbb{R}^n, \quad m \in \mathbb{N}_0,\] where we set $\Phi_{-1} \equiv 0$ to make the last equality work for $m = 0$. The $L^{1}$-bound in \eqref{BAPU_definition} follows readily from the relation \eqref{Besov_BAPU} and thus $(\psi_{m})_{m = 0}^{\infty}$ is a $\mathcal{B}(G)$-BAPU. Since the support of $\psi_{m}$ is closed and contained in $B^{\|\cdot\|}(0,2^{m+1})$, it is clear that $\psi_m$ have compact support for every $m \geq 0$.
\end{itemize}
\end{example}

Notice that the existence of a lattice in the above example is not nessesary for the case $\alpha = 1$. In fact, the $\mathcal{B}(G)$-BAPU construction is valid for any stratified Lie group regardless of its step. For the rest of the paper, we will refer to a rational stratified Lie group with step less than or equal two as an \textit{admissible Lie group} for simplicity.

\section{Generalized $\alpha$-Modulation Spaces}
\label{sec: G_alpha_modulation_spaces}

In this section, we define the generalized $\alpha$-modulation spaces $M_{p,q}^{s,\alpha}(G)$ associated to a stratified Lie group $(\mathbb{R}^n,*_G)$. The spaces $M_{p,q}^{s,\alpha}(G)$ are built on the generalized $\alpha$-coverings examined in the previous section. In many regards, the spaces $M_{p,q}^{s,\alpha}(G)$ behave similarly to their Euclidean counterparts $M_{p,q}^{s,\alpha}(\mathbb{R}^n)$. However, we will show in later sections that they depend heavily on the stratified Lie group $(\mathbb{R}^n,*_G)$ in question. Firstly, let us define the correct reservoir defining the functions/distributions of interest.

\begin{definition}
Consider the space $Z(\mathbb{R}^{n}) := \mathcal{F}\left(C_{c}^{\infty}(\mathbb{R}^n)\right)$ consisting of Fourier transforms of all smooth functions with compact support. We equip the space $Z(\mathbb{R}^{n})$ with the unique topology ensuring that the Fourier transform is a homeomorphism from $C_{c}^{\infty}(\mathbb{R}^n)$ to $Z(\mathbb{R}^{n})$. Define the \textit{Fourier type reservoir} as the dual space $Z'(\mathbb{R}^n)$ equipped with the weak$^*$ topology. 
\end{definition}

The Fourier transform extends by duality to a homeomorphism $$\mathcal{F}:Z'(\mathbb{R}^n) \longrightarrow \mathcal{D}'(\mathbb{R}^n) := \left(C_{c}^{\infty}(\mathbb{R}^n)\right)'.$$ We refer the reader to \cite[Chapter 3]{Felix_main} where the danger of using the tempered distributions $\mathcal{S}'(\mathbb{R}^n)$ as a reservoir instead of the more exotic space $Z'(\mathbb{R}^n)$ is discussed. This might seem contradictory as we defined the Euclidean $\alpha$-modulation spaces $M_{p,q}^{s,\alpha}(\mathbb{R}^n)$ in Definition \ref{usual_alpha_mod_spaces} as subspaces of the tempered distributions $\mathcal{S}'(\mathbb{R}^n)$. However, it follows from \cite[Theorem 8.3]{Felix_main} that the Euclidean $\alpha$-modulation spaces would embed into $\mathcal{S}'(\mathbb{R}^n)$ if we had defined them using the Fourier type reservoir $Z'(\mathbb{R}^n)$. Hence one might as well define the Euclidean $\alpha$ modulation spaces as subspaces of tempered distributions without loss of generality. 

\begin{definition}
\label{definition_G_alpha_modulation_spaces}
Let $(\mathbb{R}^n,*_G,\|\cdot\|)$ be a stratified Lie group with a homogeneous quasi-norm $\|\cdot\|$. Consider a generalized $\alpha$-covering $\mathcal{P}^{\alpha} = (P_{i}^{\alpha})_{i \in I}$ on $\mathbb{R}^n$ where $0 \leq \alpha \leq 1$ and assume that $\Phi = (\psi_{i})_{i \in I}$ is a $\mathcal{P}^{\alpha}$-BAPU. The \textit{generalized} $\alpha$-\textit{modulation space} $M_{p,q}^{s,\alpha}(G)$ for $1 \leq p,q \leq \infty$ and $s \in \mathbb{R}$ consists of all Fourier type distributions $f \in Z'(\mathbb{R}^{n})$ such that 
\begin{equation}
\label{alpha_modulation_norm}
    \|f\|_{M_{p,q}^{s,\alpha}(G)} := \left\|\left((1 + \|\xi_{i}\|^{2})^{\frac{s}{2}}\left\|\mathcal{F}^{-1}\left(\psi_{i} \cdot \mathcal{F}(f)\right)\right\|_{L^{p}}\right)_{i \in I}\right\|_{l^{q}(I)} < \infty,
\end{equation} where $\xi_{i} \in P_{i}^{\alpha}$ for every $i \in I$. The number $s$ will be referred to as the \textit{smoothness parameter} of the space $M_{p,q}^{s,\alpha}(G)$, while $p$ and $q$ are called the \textit{integrability parameters}.
\end{definition}

If the stratified Lie group $(\mathbb{R}^n,*_G)$ is isomorphic to the Euclidean space $(\mathbb{R}^{n},+)$ with its usual addition, then Definition \ref{definition_G_alpha_modulation_spaces} reduces to the usual $\alpha$-modulation spaces $M_{p,q}^{s,\alpha}(\mathbb{R}^n)$. Notice that \[M_{p,q}^{s + \epsilon,\alpha}(G) \subset M_{p,q}^{s,\alpha}(G), \qquad M_{p,q_1}^{s,\alpha}(G) \subset M_{p,q_2}^{s,\alpha}(G),\] for all $\epsilon > 0$ and whenever $q_1 \leq q_2$ due to the monotonicity of the  $l^{q}$-norms. It follows from Proposition \ref{proposition_about_equivalence} that any two generalized $\alpha$-coverings for the same stratified Lie group are equivalent. Hence \cite[Theorem 3.7]{Hans_Grobner} implies that $M_{p,q}^{s,\alpha}(G)$ does not depend on the specific generalized $\alpha$-covering chosen. Moreover, it follows from \cite[Theorem 2.3 B)]{Hans_Grobner} that different choices of $\mathcal{P}^{\alpha}$-BAPU's in Definition \ref{definition_G_alpha_modulation_spaces} yield equivalent norms.

\begin{remark}
As we have discussed in Subsection \ref{sec: Almost_Structured_Coverings}, we can only guarantee the existence of the BAPU's needed in Definition \ref{definition_G_alpha_modulation_spaces} in certain settings. This setting include all admissible Lie groups, which is the most interesting class when it comes to applications. However, will will state some results for generalized $\alpha$-modulation spaces on an arbitrary stratified Lie group with the convention that this might be vacuous when we do not know the existence of suitable BAPU's. In that way, some of the results we prove can still be used for a general stratified Lie group once the existence of a suitable BAPU has been established.  
\end{remark}

\begin{remark}
Let us briefly comment on why the expression \eqref{alpha_modulation_norm} is well defined: Since $f \in Z'(\mathbb{R}^n)$ we have that $\mathcal{F}(f) \in \mathcal{D}'(\mathbb{R}^n)$. Then the product $\psi_i \cdot \mathcal{F}(f)$ is a compactly supported distribution. Hence we can consider $\psi_i \cdot \mathcal{F}(f)$ as a tempered distribution and thus $\mathcal{F}^{-1}\left(\psi_{i} \cdot \mathcal{F}(f)\right) \in \mathcal{S}'(\mathbb{R}^n)$. Moreover, $\mathcal{F}^{-1}\left(\psi_{i} \cdot \mathcal{F}(f)\right)$ acts on rapidly decaying functions by integrating them against an entire function with polynomially bounded derivatives by (a variant of) the Paley-Wiener Theorem \cite[Theorem 7.23]{Rudin_Functional_Analysis}. Hence the expression in \eqref{alpha_modulation_norm} is well defined, although often infinite.
\end{remark}

The generalized $\alpha$-modulation spaces $M_{p,q}^{s,\alpha}(G)$ are complete for all values of the parameters $1 \leq p,q \leq \infty$, $s \in \mathbb{R}^n$, and $0 \leq \alpha \leq 1$ by \cite[Theorem 3.21]{Felix_main}. Motivated by the Euclidean setting, we will also refer to \[M_{p,q}^{s}(G) := M_{p,q}^{s,0}(G)\] as the \textit{modulation space} corresponding to the stratified Lie group $(\mathbb{R}^n,*_G)$. The modulation space $M_{p,q}^{s}(\mathbb{H}_n)$ corresponding to the Heisenberg group $\mathbb{H}_{n}$ has been investigated in \cite{David}. Similarly, we will also refer to \[\mathcal{B}_{p,q}^{s}(G) := M_{p,q}^{s,1}(G)\] as the \textit{Besov space} corresponding to the stratified Lie group $(\mathbb{R}^n,*_G)$. One can view the Besov spaces $\mathcal{B}_{p,q}^{s}(G)$ as generalizations of the traditional Besov spaces $\mathcal{B}_{p,q}^{s}(\mathbb{R}^n)$ where the dilations are not uniform in different directions. The spaces $M_{p,q}(G) := M_{p,q}^{0,0}(G)$ and $\mathcal{B}_{p,q}(G) := \mathcal{B}_{p,q}^{0,1}(G)$ will be called the \textit{standard modulation spaces} and \textit{standard Besov spaces} of $(\mathbb{R}^n,*_G)$, respectively. We begin by giving concrete realizations of the generalized $\alpha$-modulation spaces.

\begin{corollary}
\label{equivalent_norm}
Let $(\mathbb{R}^n, *_G, \|\cdot\|)$ be a rational stratified Lie group with a given lattice $N$. Fix the parameters $1 \leq p,q \leq \infty$ and $s \in \mathbb{R}$. 
\begin{itemize}
    \item An equivalent norm on the modulation space $M_{p,q}^{s}(G)$ is given by the expression 
    \begin{equation}
    \label{second_norm_equivalence}
        \left(\sum_{k \in N \setminus \{0\}}\left(1 + \|k\|^{2}\right)^{\frac{qs}{2}}\left\|\mathcal{F}^{-1}\left(\psi_{k} \cdot \mathcal{F}(f)\right)\right\|_{L^{p}}^{q}\right)^{\frac{1}{q}},
    \end{equation}
    where $f \in M_{p,q}^{s}(G)$ and $\{\psi_{k}\}_{k \in N \setminus \{0\}}$ is a $\mathcal{U}(G;N)$-BAPU where $\mathcal{U}(G;N)$ is described in Subsection \ref{sec: Uniform_Case}.
    \item An equivalent norm on $M_{p,q}^{s,\alpha}(G)$ for $0 < \alpha < 1$ is given by the expression 
    \begin{equation}
    \label{first_norm_equivalence}
        \left(\sum_{k \in N \setminus \{0\}}\left( 1 + \|k\|^{\frac{2}{(1-\alpha)}}\right)^{\frac{qs}{2}}\left\|\mathcal{F}^{-1}\left(\psi_{k} \cdot \mathcal{F}(f)\right)\right\|_{L^{p}}^{q}\right)^{\frac{1}{q}},
    \end{equation}
    where $f \in M_{p,q}^{s,\alpha}(G)$ and $\{\psi_{k}\}_{k \in N \setminus \{0\}}$ is a $\mathcal{Q}_{r}^{\alpha}(G;N)$-BAPU where $\mathcal{Q}_{r}^{\alpha}(G;N)$ is given in Subsection \ref{sec: Intermediate_Case}.
    \item An equivalent norm on the Besov space $\mathcal{B}_{p,q}^{s}(G)$ is given by the expression
    \begin{equation}
    \label{third_norm_equivalence}
        \left(\sum_{m = 0}^{\infty} 2^{mqs}\left\|\mathcal{F}^{-1}\left(\psi_{m} \cdot \mathcal{F}(f)\right)\right\|_{L^{p}}^{q}\right)^{\frac{1}{q}},
    \end{equation}
     where $f \in \mathcal{B}_{p,q}^{s}(G)$ and $\{\psi_{m}\}_{m= 0}^{\infty}$ is a $\mathcal{B}(G)$-BAPU where $\mathcal{B}(G)$ is described in Subsection \ref{sec: Dyadic_Case}.
\end{itemize}
If $(\mathbb{R}^n, *_G, \|\cdot\|)$ is an admissible Lie group, then we can pick the explicit BAPU's given in Example \ref{BAPU_example}.
\end{corollary}

\begin{proof}
Consider the intermediate case $0 < \alpha < 1$: Picking the center point $\delta_{\beta}(k)$ in each of the balls in the covering $\mathcal{Q}_{r}^{\alpha}(G;N)$ gives \[\left(1 + \|\delta_{\beta}(k)\|^{2}\right)^{\frac{1}{2}} = \left(1 + \|k\|^{2(\beta + 1)}\right)^{\frac{1}{2}} = \left(1 + \|k\|^{\frac{2}{1-\alpha}}\right)^{\frac{1}{2}}, \quad k \in N \setminus \{0\}.\] Hence \eqref{first_norm_equivalence} follows and we obtain \eqref{second_norm_equivalence} automatically since $\mathcal{U}(G;N) = \mathcal{Q}_{r}^{0}(G;N)$ as explained in Proposition \ref{explicit_alpha_covering_prop}. For the Besov covering $\mathcal{B}(G)$ the choice of homogeneous quasi-norm in \eqref{alpha_modulation_norm} is irrelevant due to Lemma \ref{Besov_covering_lemma}. Hence we can freely choose the homogeneous quasi-norm $\|\cdot\|_{2}$ given in \eqref{spesific_homogeneous_quasi_norm} and use consequence \eqref{elemetns_in_dyadic_covering}. Thus we obtain\[\left(1 + \|\left(2^{m}, 0, \dots ,0\right)\|_{2}^{2}\right)^{\frac{1}{2}} = \left(1 + 4^m\right)^{\frac{1}{2}} \asymp 2^m, \quad m \in \mathbb{N}_{0},\] and the statement follows. 
\end{proof}

\begin{remark}
We would like to emphasize that the expression \eqref{third_norm_equivalence} does not depend on the lattice $N$ and is hence valid whenever $(\mathbb{R}^n,*_G)$ is not rational. The expressions \eqref{third_norm_equivalence} also shows the similarities with the classical Besov spaces $\mathcal{B}_{p,q}^{s}(\mathbb{R}^n)$ in the literature. 
\end{remark}

Let us introduce some notation that simplifies the expressions in Corollary \ref{equivalent_norm} for admissible Lie groups when $0 \leq \alpha < 1$: Consider the generalized $\alpha$-covering $\mathcal{Q}^{\alpha}(G;N)$ and a smooth $\mathcal{Q}^{\alpha}(G;N)$-BAPU $\{\psi_{k}\}_{k \in N \setminus \{0\}}$ with compact support. Let $\mathcal{P}(\mathbb{R}^n;N)$ denote sequences $\{f_k\}_{k \in N \setminus \{0\}}$ where each $f_k \in \mathcal{S}'(\mathbb{R}^n)$ acts on rapidly decaying functions by integrating against a polynomially bounded function. Define the \textit{Fourier multiplier operator} \[\square_{k}^{\alpha,G} := \mathcal{F}^{-1}\psi_{k}\mathcal{F},\] and the space \[l_{p,q}^{s,\alpha}(G;N) := \left\{\{f_k\}_{k \in N \setminus \{0\}} \in  \mathcal{P}(\mathbb{R}^n;N) \, \Big| \, \left\|\left(1+ \|k\|^{\frac{2}{1-\alpha}}\right)^{\frac{s}{2}}\|f_k\|_{L^p} \right\|_{l^{q}(N \setminus \{0\})} < \infty \right\}.\] The generalized $\alpha$-modulation space $M_{p,q}^{s,\alpha}(G)$ for an admissible Lie group $(\mathbb{R}^n,*_G)$ and parameters $0 \leq \alpha < 1$, $ 1 \leq p,q \leq \infty$ and $s \in \mathbb{R}$ can be written as 
\begin{equation}
\label{other_description_G_alpha_mod_spaces}
M_{p,q}^{s,\alpha}(G) = \left\{f \in Z'(\mathbb{R}^n) \, \Big| \, \|\square_{k}^{\alpha,G}f\|_{l_{p,q}^{s,\alpha}(G;N)} < \infty \right\}.
\end{equation}
We emphasize that this is only valid since the image of the Fourier type distributions $Z'(\mathbb{R}^n)$ under the Fourier multiplier operator $\square_{k}^{\alpha,G}$ for $k \in N \setminus \{0\}$ and $0 \leq \alpha < 1$ is contained in the tempered distributions. \par
The representation \eqref{other_description_G_alpha_mod_spaces} is useful because it allows us to reduce certain questions about the generalized $\alpha$-modulation spaces $M_{p,q}^{s,\alpha}(G)$ to the sequence-type spaces $l_{p,q}^{s,\alpha}(G;N)$. As an application, we now prove a duality relation between the spaces $M_{p,q}^{s,\alpha}(G)$ for $1 \leq p,q < \infty$, $s \in \mathbb{R}$ and $0 \leq \alpha < 1$. This is more or less a straightforward adaption of the Euclidean case given in \cite[Theorem 2.1]{Han} with some minor modifications resulting from using an arbitrary lattice $N$ instead of the concrete lattice $\mathbb{Z}^{n} \subset \mathbb{R}^n$. 

\begin{proposition}
Let $(\mathbb{R}^n,*_G)$ be an admissible Lie group and fix parameters $1 \leq p,q < \infty$, $s \in \mathbb{R}$ and $0 \leq \alpha < 1$. The dual space of $M_{p,q}^{s,\alpha}(G)$ can be identified with $M_{p',q'}^{-s,\alpha}(G)$, where $p'$ and $q'$ are the conjugate variables of $p$ and $q$, respectively. 
\end{proposition}

\begin{proof}
Fix a lattice $N \subset (\mathbb{R}^n,*_G)$. Any $f \in M_{p',q'}^{-s,\alpha}(G)$ acts on elements $g \in M_{p,q}^{s,\alpha}(G)$ by \[\langle f,g\rangle := \sum_{k \in N \setminus \{0\}}\int_{\mathbb{R}^n}\square_{k}^{\alpha,G}f \cdot \square_{k}^{\alpha,G}g \, dx.\]
A straightforward computation using H\"{o}lder's inequality twice shows that 
\begin{align*}
    |\langle f,g\rangle| & = \left|\sum_{k \in N\setminus \{0\}}\int_{\mathbb{R}^n}\left(1 + \|k\|^{\frac{2}{1- \alpha}}\right)^{\frac{-s}{2}}\square_{k}^{\alpha,G}f \cdot \left(1 + \|k\|^{\frac{2}{1- \alpha}}\right)^{\frac{s}{2}}\square_{k}^{\alpha,G}g \, dx\right| \\ & \leq
    \left|\sum_{k \in N \setminus \{0\}}\left(\int_{\mathbb{R}^n}\left(1 + \|k\|^{\frac{2}{1- \alpha}}\right)^{\frac{-sp'}{2}}[\square_{k}^{\alpha,G}f]^{p'} \, dx\right)^{\frac{1}{p'}} \cdot \left(\int_{\mathbb{R}^n}\left(1 + \|k\|^{\frac{2}{1- \alpha}}\right)^{\frac{sp}{2}}[\square_{k}^{\alpha,G}g]^{p} \, dx\right)^{\frac{1}{p}} \right| \\ & \leq \|f\|_{M_{p',q'}^{-s,\alpha}(G)}\|g\|_{M_{p,q}^{s,\alpha}(G)},
\end{align*}
where we have used the explicit expressions in \eqref{first_norm_equivalence} and \eqref{second_norm_equivalence}. Hence the action of $f$ on $M_{p,q}^{s,\alpha}(G)$ is bounded and we have $M_{p',q'}^{-s,\alpha}(G) \subset (M_{p,q}^{s,\alpha}(G))^{*}$. \par 
Conversely, an element $h \in (M_{p,q}^{s,\alpha}(G))^{*}$ induce an element $\tilde{h} \in \left(l_{p,q}^{s,\alpha}(G;N)\right)^{*}$ since we can identify $M_{p,q}^{s,\alpha}(G)$ with the image $\{\square_{k}^{\alpha,G}f\}_{k \in N \setminus \{0\}} \in l_{p,q}^{s,\alpha}(G;N).$
A standard argument similar to the one given in \cite[Proposition 3.3]{Wang} shows that we have the duality relation \[\left(l_{p,q}^{s,\alpha}(G;N)\right)^{*} \simeq l_{p',q'}^{-s,\alpha}(G;N), \qquad \frac{1}{p} + \frac{1}{p'} = \frac{1}{q} + \frac{1}{q'} = 1.\]
Hence we can find $\{h_k\} \in l_{p',q'}^{-s,\alpha}(G;N)$ such that \[\langle \tilde{h},\{f_k\}\rangle = \sum_{k \in N \setminus \{0\}}\int_{\mathbb{R}^n}\overline{h_{k}(x)}f_{k}(x) \, dx, \quad \{f_k\} \in l_{p,q}^{s,\alpha}(G;N).\] Thus for $g \in M_{p,q}^{s,\alpha}(G)$ we can use Plancherel to obtain \[\langle h,g\rangle = \langle \tilde{h}, \square_{k}^{\alpha,G}g\rangle = \int_{\mathbb{R}^n}\sum_{k \in N \setminus \{0\}}\overline{\square_{k}^{\alpha,G}h_{k}(x)}g(x)\, dx,\] and we can conclude that \[h = \sum_{k \in N \setminus \{0\}}\square_{k}^{\alpha,G}h_k.\]
A straightforward generalization of \cite[Lemma 2.1]{Han} shows that \[\|h\|_{M_{p',q'}^{-s,\alpha}(G)} \simeq \|\{h_k\}_{k \in N \setminus \{0\}}\|_{l_{p',q'}^{-s,\alpha}(G;N)} = \|h\|_{(M_{p,q}^{s,\alpha}(G))^{*}},\] implying that $(M_{p,q}^{s,\alpha}(G))^{*} \subset M_{p',q'}^{-s,\alpha}(G)$.
\end{proof}

The duality relations for the Besov spaces are the obvious extensions of the Euclidean Besov spaces, namely \[\left(\mathcal{B}_{p,q}^{s}(G)\right)^{*} \simeq B_{p',q'}^{-s}(G),\] where $1 \leq p,q < \infty$, $s \in \mathbb{R}$, and $p',q'$ are the conjugate variables of $p$ and $q$. We omit the proof since it is straightforward.

\begin{corollary}
Let $(\mathbb{R}^n,*_G)$ be an admissible Lie group. The modulation spaces $M_{p,q}^{s}(G)$ and Besov spaces $\mathcal{B}_{p,q}^{s}(G)$ for $1 < p,q < \infty$ and $s \in \mathbb{R}$ are two families of reflexive Banach spaces spaces that are closed under duality.
\end{corollary}

There is an abundance of properties one can prove when defining new function spaces. We will focus on two properties illustrating that the generalized $\alpha$-modulation spaces $M_{p,q}^{s,\alpha}(G)$ are not degenerate:
\begin{enumerate}[label=(\roman*)]
\label{enumerate_nondegenerate_properties}
    \item When $n = \textrm{dim}(G)$ the spaces $M_{p,q}^{s,\alpha}(G)$ are large enough to contain the rapidly decaying smooth functions $\mathcal{S}(\mathbb{R}^n)$ as subspaces.
    \item The spaces $M_{p_1,q_1}^{s_1,\alpha_1}(G)$ are really new spaces in the sense that they do not coincide with the traditional $\alpha$-modulation spaces $M_{p_2,q_2}^{s_2,\alpha_2}(\mathbb{R}^n)$ for most of the parameters $1 \leq p_1,p_2,q_1,q_2 \leq \infty$, $s_1,s_2 \in \mathbb{R}$, and $0 \leq \alpha_1,\alpha_2 \leq 1$.
\end{enumerate}
Property \hyperref[enumerate_nondegenerate_properties]{(i)} relies on basic properties of lattices in stratified Lie groups and is proved below. On the other hand, Property \hyperref[enumerate_nondegenerate_properties]{(ii)} is more challenging and require several preliminary results. The main aim of Section \ref{sec: Novelty} is to prove Property \hyperref[enumerate_nondegenerate_properties]{(ii)}.

\begin{proposition}
\label{Schwartz_functions}
Let $(\mathbb{R}^n,*_G)$ be a rational stratified Lie group with $\textrm{dim}(G) = n$. Then the rapidly decaying smooth functions $\mathcal{S}(\mathbb{R}^n)$ is contained in $M_{p,q}^{s,\alpha}(G)$ for all $1 \leq p,q \leq \infty$, $s \in \mathbb{R}$ and $0 \leq \alpha \leq 1$. 
\end{proposition}
\begin{proof}
The embeddings $M_{p,q_1}^{s,\alpha}(G) \subset M_{p,q_2}^{s,\alpha}(G)$ for $q_1 \leq q_2$ implies that it suffices to show the inclusion $\mathcal{S}(\mathbb{R}^n) \subset M_{p,1}^{s,\alpha}(G)$. We consider first the case $0 \leq \alpha < 1$ and use the norms \eqref{second_norm_equivalence} and \eqref{first_norm_equivalence}. Since the Fourier transform of a Schwartz function is again a Schwartz function, we have that \[\epsilon_k := \left\|\mathcal{F}^{-1}\left(\psi_{k} \cdot \mathcal{F}(f)\right)\right\|_{L^{p}} < \infty, \quad k \in N\setminus \{0\}\] for each $f \in \mathcal{S}(\mathbb{R}^n)$ since $\mathcal{S}(\mathbb{R}^n) \subset L^{p}$ for every $1 \leq p \leq \infty$. \par  Fix a homogeneous norm $\|\cdot\|$ on $(\mathbb{R}^n,*_G)$ as we have remarked previously always exists and define the metric 
\begin{equation}
\label{metric_from_homogeneous_norm}
    d_{G}(x,y) := \|x^{-1} *_{G} y\|, \quad  x,y \in \mathbb{R}^n.
\end{equation} 
We claim that the number of points in $N$ that are of distance less than $R > 0$ away from the origin with respect to the metric $d_{G}$ grows with polynomial rate. It is clear that $d_{G}$ is a left-invariant metric on $(\mathbb{R}^n,*_G)$. It restricts to a proper, left-invariant metric $d_{G}|_N$ on $N$. Since $N$ is a finitely generated nilpotent group we know that $d_{G}|_N$ has polynomial growth by Gromov's celebrated polynomial growth theorem \cite{Polynomial_Growth_Paper} and the claim follows.  \par 
It is straightforward to see that the numbers $\epsilon_k$ decay exponentially as the size of $k \in N \setminus \{0\}$ grows with respect to the metric $d_{G}|_N$. The weight \[\left( 1 + \|k\|^{\frac{2}{(1-\alpha)}}\right)^{\frac{s}{2}}\] only contributes polynomially since both $\alpha$ and $s$ are fixed. Hence it follows that $\mathcal{S}(\mathbb{R}^n) \subset M_{p,1}^{s,\alpha}(G) \subset M_{p,q}^{s,\alpha}(G)$ since $l^{1}(N \setminus \{0\})$ contains all rapidly decreasing sequences. The case $\alpha = 1$ is an elementary adaption of the classical proof of the inclusion $\mathcal{S}(\mathbb{R}^n) \subset \mathcal{B}_{p,q}^{s}(\mathbb{R}^n)$. 
\end{proof}

\section{Uniqueness of the Generalized $\alpha$-Modulation Spaces}
\label{sec: Novelty}

\subsection{Preliminary Results}

We now turn to the question regarding the uniqueness of the spaces $M_{p,q}^{s,\alpha}(G)$ for an admissible Lie group $(\mathbb{R}^n,*_G)$. To be able to answer this, we need a stronger statement about the generalized $\alpha$-coverings than what was proved in Proposition \ref{proposition_about_equivalence}. Parts of the Euclidean case (when $G = (\mathbb{R}^n,+)$) of Theorem \ref{weakly_subordinate_result} has been proved in \cite[Lemma 9.5 and Lemma 9.12]{Felix_main} with different methods. We remark that the proof of Theorem \ref{weakly_subordinate_result} does build on Proposition \ref{proposition_about_equivalence} and most of Section \ref{sec: G_alpha_coverings} in a non-trivial way.

\begin{theorem}
\label{weakly_subordinate_result}
Let $(\mathbb{R}^n,*_G)$ be a rational stratified Lie group and consider a generalized $\alpha_1$-covering $\mathcal{P}^{\alpha_1}$ and a generalized $\alpha_2$-covering $\mathcal{P}^{\alpha_2}$ for some $0 \leq \alpha_1, \alpha_2 \leq 1$. Then
\[\left\{\textrm{$\mathcal{P}^{\alpha_1}$ is weakly subordinate to $\mathcal{P}^{\alpha_2}$}\right\} \Longleftrightarrow \left\{\textrm{$\mathcal{P}^{\alpha_1}$ is almost subordinate to $\mathcal{P}^{\alpha_2}$}\right\} \Longleftrightarrow \left\{\alpha_1 \leq \alpha_2\right\}.\]
In particular, the coverings $\mathcal{P}^{\alpha_1}$ and $\mathcal{P}^{\alpha_2}$ are equivalent if and only if $\alpha_1 = \alpha_2$.
\end{theorem}

\begin{proof}
Since all the coverings in question consist of open and connected sets, it suffices to show the second equivalence as pointed out previously. Moreover, since we have proved in Proposition \ref{proposition_about_equivalence} that any two generalized $\alpha$-coverings are equivalent, it suffices to consider the explicit coverings $\mathcal{Q}^{\alpha_1}(G)$ and $\mathcal{Q}^{\alpha_2}(G)$ in \eqref{explicit_coverings_cases}. Let us fix a lattice $N$ in $(\mathbb{R}^n,*_G)$ and a homogeneous quasi-norm $\|\cdot\|$. \par 
We begin by considering the Besov case $\alpha_2 = 1$. We note that the size of the sets $D_{m}(G)$ grows exponentially with respect to $m \in \mathbb{N}_0$. However, the size of the elements in $\mathcal{Q}^{\alpha_1} = (Q_{n}^{\alpha_1}(G))_{n \in N \setminus \{0\}}$ for $0 \leq \alpha_1 < 1$ grows polynomially when we order the index set $N \setminus \{0\}$ in a way such that $m \leq n$ whenever $\|m\| \leq \|n\|$. Hence the number \[\#\left\{n \in N \setminus \{0\} \, \Big| \, D_{m}(G) \cap Q_{n}^{\alpha_1}(G) \neq \emptyset \right\}\] will grow unbounded as $m$ increases. This shows that $\mathcal{B}(G)$ is not weakly subordinate to any $\mathcal{Q}^{\alpha_1}(G)$ for $0 \leq \alpha_1 < 1$. \par 
Next, we need to show that $\mathcal{Q}^{\alpha_1}(G)$ is weakly subordinate to the Besov covering $\mathcal{B}(G)$ whenever $0 \leq \alpha_1 < 1$. Pick the center point \[\delta_{\beta}(n) = n\|n\|^{\frac{\alpha_1}{1-\alpha_1}} \in Q_{n}^{\alpha_1}(G) := B^{\|\cdot\|}(\delta_{\beta}(n),r\|n\|^{\beta}), \quad \beta :=\frac{\alpha_1}{1 - \alpha_1}.\] Then if $2^{m-1} \leq \|\delta_{\beta}(n)\| \leq 2^{m+1}$ for some $m \in \mathbb{N}$ we have that $2^{(m-1)(1-\alpha_1)} \leq \|n\| \leq 2^{(m+1)(1-\alpha_1)}$. Hence for all $n \in N \setminus \{0\}$ where $\|n\|$ is sufficiently large it follows that \[\# \left\{m \in \mathbb{N}_{0} \, \Big| \, Q_{n}^{\alpha_1}(G) \cap D_{m}(G) \neq \emptyset \right\} \in \{1,2\}.\] Since there are only a finite number of elements $n \in N \setminus \{0\}$ with norm less than a fixed tolerance, we have that $\mathcal{Q}^{\alpha_1}(G)$ is weakly subordinate to the Besov covering $\mathcal{B}(G)$. Showing that $\mathcal{Q}^{\alpha_1}(G)$ is weakly subordinate to $\mathcal{Q}^{\alpha_2}(G)$ whenever $0 \leq \alpha_1 \leq \alpha_2 < 1$ is straightforward since the function 
\begin{equation}
\label{increasing_function}
    x \longmapsto \frac{x}{1 - x}, \qquad x \in [0,1),
\end{equation}
is increasing. It only remains to show that $\mathcal{Q}^{\alpha_2}(G)$ is not weakly subordinate to $\mathcal{Q}^{\alpha_1}(G)$ when $0 \leq \alpha_1 < \alpha_2 < 1$. \par 
We use the notation \[A_{\alpha_1}^{\alpha_2}(n) := \# \left\{l \in N \setminus \{0\} \, \Big| \, Q_{n}^{\alpha_2}(G) \cap Q_{l}^{\alpha_1}(G) \neq \emptyset \right\},\] and will give an iterated argument to show that there is no uniform bound on $A_{\alpha_1}^{\alpha_2}(n)$ for all $n \in N \setminus \{0\}$. The size of the elements in $\mathcal{Q}^{\alpha_1}(G)$ is given by \[| Q_{n}^{\alpha_1}(G)| = \Big| B^{\|\cdot\|}\left(\delta_{\beta}(n),r\|n\|^{\beta}\right)\Big| = \left(r\|n\|^{\frac{\alpha_1}{1 - \alpha_1}}\right)^{Q}\cdot \mu, \qquad \mu := \Big|B^{\|\cdot\|}(0,1)\Big|,\] where $Q$ is the homogeneous dimension of $(\mathbb{R}^n,*_G)$. Since the function given in \eqref{increasing_function} is increasing there exists for every $\epsilon > 0$ a threshold $R > 0$ such that for $\|n\| \geq R$ we have \[\frac{| Q_{n}^{\alpha_1}(G)|}{| Q_{n}^{\alpha_2}(G)|} = \|n\|^{Q\left(\frac{\alpha_1}{1 - \alpha_1} - \frac{\alpha_2}{1 - \alpha_2}\right)} < \epsilon.\] The fact that $N$ is a uniform lattice gives that there is a number $C_{N} > 0$ such that $\|m^{-1}n\| \leq C_{N}$ for every $m \in N \setminus \{0\}$ that is a neighbour of $n$ with respect to the covering $\mathcal{Q}^{\alpha_2}(G)$. Hence \[|Q_{n}^{\alpha_2}(G)| \simeq |Q_{m}^{\alpha_2}(G)|,\] for all such $m$. Therefore we can, by increasing the threshold $R$, find a sequence $n_k \in N \setminus \{0\}$ such that $A_{\alpha_1}^{\alpha_2}(n_k) \to \infty$. This implies that $\mathcal{Q}^{\alpha_2}$ is not weakly subordinate to $\mathcal{Q}^{\alpha_1}(G)$ whenever $0 \leq \alpha_1 < \alpha_2 < 1$. 
\end{proof}

Before we turn to the uniqueness result we need to investigate the generalized $\alpha$-coverings in the two extreme cases $\alpha \in \{0,1\}$ more thoroughly. To do this, we will briefly review a procedure originating in \cite{Hans_Grobner} and more recently investigated in \cite{ThesisRene} and \cite{Berge} that associates a metric space to any concatenation. Although investigating the extreme cases $\alpha \in \{0,1\}$ could be done without this extra machinery, we will need this approach in Section \ref{sec: Metric_Geometry_G_alpha_Modulation_Spaces} anyway and hence go through the necessary definitions here. \par
Associated to any concatenation $\mathcal{Q}$ on $\mathbb{R}^n$ is a metric $d_{\mathcal{Q}}$ on $\mathbb{R}^n$ that reflects the global properties of $\mathcal{Q}$. For two distinct points $x,y \in \mathbb{R}^n$ we define the distance $d_{\mathcal{Q}}(x,y)$ to be the minimal number $k$ such that there is a sequence $Q_{i_1}, \dots, Q_{i_k}$ connecting $x$ and $y$. To be more formal, we require that $x \in Q_{i_1}$, $y \in Q_{i_k}$ and that $Q_{i_j} \cap Q_{i_{j+1}} \neq \emptyset$ for every $j = 1, \dots, k-1$ and that no such sequence of length $k-1$ exists. We extend the definition by $d_{\mathcal{Q}}(x,x) = 0$ for all $x \in \mathbb{R}^n$ and refer to $(\mathbb{R}^n,d_{\mathcal{Q}})$ as the \textit{associated metric space} to the concatenation $\mathcal{Q}$. Comparing two metric spaces $(\mathbb{R}^n,d_{\mathcal{Q}})$ and $(\mathbb{R}^m,d_{\mathcal{P}})$ corresponding to different coverings $\mathcal{Q}$ and $\mathcal{P}$ is done by employing the notion of quasi-isometric embeddings. 
\begin{definition}
A \textit{quasi-isometric embedding} between $(\mathbb{R}^n,d_{\mathcal{Q}})$ and $(\mathbb{R}^m,d_{\mathcal{P}})$ is a map $f:(\mathbb{R}^n,d_{\mathcal{Q}}) \to (\mathbb{R}^m,d_{\mathcal{P}})$ with fixed parameters $L,C > 0$ such that \[ \frac{1}{L}d_{\mathcal{Q}}(x,y) - C \leq d_{\mathcal{P}}(f(x),f(y)) \leq Ld_{\mathcal{Q}}(x,y) + C,\] for all $x,y \in \mathbb{R}^{n}$. We say that $f$ is a \textit{quasi-isometry} if additionally $f(\mathbb{R}^n)$ is a \textit{net} in $\mathbb{R}^m$, that is, we have the uniform bound \[\adjustlimits \sup_{y \in \mathbb{R}^m} \inf_{x \in f(\mathbb{R}^n)} d_{\mathcal{P}}\left(x,y\right) < \infty.\]
\end{definition}

\begin{proposition}
\label{Hans_result}
Let $\mathcal{Q}$ and $\mathcal{P}$ be two concatenations on $\mathbb{R}^n$. Then $\mathcal{Q}$ is almost subordinate to $\mathcal{P}$ if and only if the identity map $Id:(\mathbb{R}^n,d_{\mathcal{Q}}) \to (\mathbb{R}^n,d_{\mathcal{P}})$ is Lipschitz continuous. Moreover, the concatenations $\mathcal{Q}$ and $\mathcal{P}$ are equivalent if and only if the identity map $Id:(\mathbb{R}^n,d_{\mathcal{Q}}) \to (\mathbb{R}^n,d_{\mathcal{P}})$ is a quasi-isometry.
\end{proposition}

The proposition above originates from \cite[Proposition 3.8]{Hans_Grobner} and was phrased in the language of quasi-isometries first for open sets of a Euclidean space in \cite{ThesisRene} and for more general coverings in \cite{Berge}. It shows that the metric space approach extends the notion of almost subordination to coverings defined on different Euclidean spaces. We now use the metric space viewpoint of coverings to examine the boundary cases $\alpha \in \{0,1\}$. 

\begin{proposition}
\label{Besov_are_all_the_same}
Consider the rational stratified Lie group $(\mathbb{R}^n,*_G)$. The Besov covering $\mathcal{B}(G)$ is equivalent to the Euclidean Besov covering $\mathcal{B}(\mathbb{R}^n)$ only when $(\mathbb{R}^n,*_G)$ is isomorphic to $(\mathbb{R}^n,+)$. Similarly, the uniform covering $\mathcal{U}(G)$ is only equivalent to the Euclidean uniform covering $\mathcal{U}(\mathbb{R}^n)$ when $(\mathbb{R}^n,*_G)$ is isomorphic to $(\mathbb{R}^n,+)$.
\end{proposition}

\begin{proof}
We can by Lemma \ref{Besov_covering_lemma} choose to work with the homogenous quasi-norm $\|\cdot\|_{2}$ given in  \eqref{spesific_homogeneous_quasi_norm}. It is straightforward to check that the points $p(m) := (2^{m},0,\dots,0)$ and $q(m) := (0,\dots, 2^{sm+1})$ are both in $D_{m}(G)$ for all $m \in \mathbb{N}$, where $s$ is the step of $(\mathbb{R}^n,*_G)$. Hence $d_{\mathcal{B}(G)}(p(m),q(m)) = 1$. However, the distance $d_{\mathcal{B}(\mathbb{R}^n)}(p(m),q(m))$ tends to infinity as $m$ increases as long as $s > 1$. Hence the identity map \[Id:(\mathbb{R}^n,d_{\mathcal{B}(G)}) \to (\mathbb{R}^n,d_{\mathcal{B}(\mathbb{R}^n)})\] is not a quasi-isometry and we can apply Proposition \ref{Hans_result} to obtain that the concatenations $\mathcal{B}(G)$ and $\mathcal{B}(\mathbb{R}^n)$ are not equivalent. When $s = 1$ there is only one layer in the stratification \eqref{Lie_algebra_decomposition} and it is clear that $(\mathbb{R}^n,*_G) \simeq (\mathbb{R}^n,+)$ in that case.\par 
The second statement follows from the more general statement proved in \cite[Theorem 3.6]{Berge} implying that the uniform coverings $\mathcal{U}(G)$ and $\mathcal{U}(H)$ of two rational stratified Lie groups $(\mathbb{R}^n,*_G)$ and $(\mathbb{R}^n,*_H)$ can only be equivalent if the groups have the same homogeneous dimension. The homogeneous dimension $Q$ of $(\mathbb{R}^n,*_G)$ satisfies $Q = n$ only when the stratification \eqref{Lie_algebra_decomposition} has only one layer. Hence we conclude that $(\mathbb{R}^n,*_G) \simeq (\mathbb{R}^n,+)$ and the result follows.
\end{proof}

\begin{remark}
The Besov covering $\mathcal{B}(G)$ of a stratified Lie group $(\mathbb{R}^n,*_G)$ fits in a larger class of coverings investigated in \cite{Hartmut_Besov} known as inhomogeneous covering induced by an expansive matrix. An \textit{expansive matrix} $A$ is a matrix such that all its eigenvalues have norm strictly greater than one. Consider a collection $\mathcal{C} = (C_{j})_{j \in \mathbb{N}_0}$ such that $C_0$ and $C_1$ are the closures of two bounded and open sets and $C_{j} = A^{j-1}(C_1)$ for $j \geq 1$. If \[\bigcup_{j = 0}^{\infty}C_{j} = \mathbb{R}^n,\] then the collection $\mathcal{C}$ is called an \textit{inhomogeneous covering} induced by the expansive matrix $A$. For our Besov coverings $\mathcal{B}(G)$, we have $C_0 = \overline{D_{0}(G)}$, $C_1 = \overline{D_{1}(G)}$, and \[A = \begin{pmatrix}
    2^{v_1} & & \\
    & \ddots & \\
    & & 2^{v_n}
\end{pmatrix}, \qquad v_j := \textrm{deg}(x_j).\]
The attentive reader will notice that there is a small discrepancy as the sets in the covering $\mathcal{B}(G)$ are open, while the one described above consists of the closures of these elements. This is of minor importance as the two versions are clearly equivalent.  Written in this framework, we could use \cite[Lemma 5.7 (b)]{Hartmut_Besov} to derive the first statement in Proposition \ref{Besov_are_all_the_same}.
\end{remark}

We need a final lemma regarding the weights appearing in Corollary \ref{equivalent_norm} before answering the uniqueness of the generalized $\alpha$-modulation spaces $M_{p,q}^{s,\alpha}(G)$. 

\begin{lemma}
\label{lemma_moderate_weights}
Let $(\mathbb{R}^n,*_G)$ be an admissible Lie group with a lattice $N$ and denote by $\mathcal{Q}^{\alpha}(G)$ the explicit generalized $\alpha$-coverings given in \eqref{explicit_coverings_cases}. The weights \[N \setminus \{0\} \ni k \longmapsto \left(1 + \|k\|^{\frac{2}{1-\alpha}}\right)^{\frac{s}{2}}, \quad 0 \leq \alpha < 1,\] are $\mathcal{Q}^{\alpha}(G)$-moderate and the weight \[\mathbb{N}_{0} \ni m \longmapsto 2^{ms}\] is $\mathcal{B}(G)$-moderate.
\end{lemma}

\begin{proof}
For the Besov case, recall that two elements $D_{n}(G)$ and $D_{m}(G)$ for $n,m \in \mathbb{N}_0$ only intersect whenever $m \in \{n-1,n,n+1\}.$ Hence the weight is $\mathcal{B}(G)$-moderate since $2^{(n+1)s}/2^{ns} =  2^{ns}/2^{(n-1)s} = 2^s$. \par
Let us consider the case $\alpha = 0$; we omit the more cumbersome case $0 < \alpha < 1$ as it relies on the same idea along with computations that can be found in the proof of \cite[Lemma 9.2]{Felix_main}. Assume that $B^{\|\cdot\|}(k,R) \cap B^{\|\cdot\|}(l,R) \neq \emptyset$ for $k,l \in N \setminus \{0\}$ where $R > 0$ is large enough so that $\mathcal{U}(G;N)$ is a covering. Then the triangle inequality implies that $\|l^{-1}*_G k\| \leq 2CR$ where $C \geq 1$ is the quasi-norm constant in \eqref{quasi_norm_constant}. Hence we obtain \[\frac{\omega(k)}{\omega(l)} = \left(\frac{1 + \|l *_G l^{-1} *_G k\|^2}{1 + \|l\|^2}\right)^{\frac{s}{2}} \leq \left(\frac{1 + (\|l\| + 2R)^2}{1 + \|l\|^2}\right)^{\frac{s}{2}} = \left(1 + 4R^2\frac{1 + R^{-1}\|l\|}{1 + \|l\|^2}\right)^{\frac{s}{2}}\leq \mathcal{C}_{\omega},\] where the constant $\mathcal{C}_{\omega}$ does not depend on the choice of lattice points $l,k \in N \setminus \{0\}$.
\end{proof}

\subsection{Main Result}

We now have all the tools needed to answer the question regarding uniqueness of the spaces $M_{p,q}^{s,\alpha}(G)$. For the case $G = \mathbb{H}_{n}$ and $\alpha = 0$, this question has been settled in \cite[Theorem 7.6]{David}. The authors showed that $M_{p_1,q_1}^{s_1}(\mathbb{H}_n) \neq M_{p_2,q_2}^{s_2}(\mathbb{R}^{2n+1})$ unless $(p_1,q_1,s_1) = (p_2,q_2,s_2) = (2,2,0)$, in which case \[M_{2,2}^{0}(\mathbb{H}_n) = M_{2,2}^{0}\left(\mathbb{R}^{2n+1}\right) = L^{2}\left(\mathbb{R}^{2n+1}\right).\] We say that the parameters $1 \leq p,q \leq \infty$, $s \in \mathbb{R}$, and $0 \leq \alpha \leq 1$ are \textit{non-trivial} if $(p,q,s) \neq (2,2,0)$. We are now ready to state the uniqueness result.

\begin{theorem}
\label{novelty_theorem}
Consider an admissible stratified Lie group $(\mathbb{R}^n,*_G)$ and two sets of non-trivial parameters $1 \leq p_1,p_2,q_1,q_2 \leq \infty$, $s_1,s_2 \in \mathbb{R}$, and $0 \leq \alpha_1,\alpha_2 \leq 1$. We have equality \[M_{p_1,q_1}^{s_1,\alpha_1}(G) = M_{p_2,q_2}^{s_2,\alpha_2}(\mathbb{R}^n)\] with equivalent norms if and only if both \[(p_1,q_1,s_1,\alpha_1) = (p_2,q_2,s_2,\alpha_2) \quad \text{and} \quad (\mathbb{R}^n,*_G) \simeq (\mathbb{R}^n,+).\]
\end{theorem}

\begin{proof}
Assume first that all the data coincide, that is, $(p_1,q_1,s_1,\alpha_1) = (p_2,q_2,s_2,\alpha_2)$ and $(\mathbb{R}^n,*_G) \simeq (\mathbb{R}^n,+)$. We can apply Proposition \ref{proposition_about_equivalence} to obtain that $\mathcal{Q}^{\alpha_1}(G)$ is equivalent to $\mathcal{Q}^{\alpha_2}(\mathbb{R}^n)$. Then the first implication follows from \cite[Theorem 3.7]{Hans_Grobner} stating that two decomposition spaces are equal with equivalent norms whenever we have equivalent underlying coverings and equal parameters. The difficult part is the converse, and the rest of the proof is devoted to this direction. \par
Assume that we have equality $M_{p_1,q_1}^{s_1,\alpha_1}(G) = M_{p_2,q_2}^{s_2,\alpha_2}(\mathbb{R}^n)$ with equivalent norms. We start by applying the very general result \cite[Theorem 6.9]{Felix_main} implying that $(p_1,q_1) = (p_2,q_2)$ and that the coverings $\mathcal{Q}^{\alpha_1}(G)$ and $\mathcal{Q}^{\alpha_2}(\mathbb{R}^n)$ are weakly equivalent. We needed Lemma \ref{lemma_moderate_weights} to invoke this result. Since both coverings consist of open and path-connected sets, it follows that $\mathcal{Q}^{\alpha_1}(G)$ and $\mathcal{Q}^{\alpha_2}(\mathbb{R}^n)$ are equivalent. Our strategy to show equality of the smoothness parameters $s_1,s_2$ is to use the result \cite[Theorem 6.9 (4b)]{Felix_main} showing that the weights $\omega_{\alpha_1,G}$ and $\omega_{\alpha_2,\mathbb{R}^n}$ corresponding to the coverings $\mathcal{Q}^{\alpha_1}(G) = (Q_{i}^{\alpha_1}(G))_{i \in I}$ and $\mathcal{Q}^{\alpha_2}(\mathbb{R}^n) = (Q_{j}^{\alpha_2}(\mathbb{R}^n))_{j \in J}$ are equivalent whenever $M_{p_1,q_1}^{s_1,\alpha_1}(G) = M_{p_2,q_2}^{s_2,\alpha_2}(\mathbb{R}^n)$ in the sense that there exists a constant $C > 0$ such that 
\begin{equation}
\label{equivalent_weights}
    \frac{1}{C}\omega_{\alpha_1,G}(i) \leq \omega_{\alpha_2,\mathbb{R}^n}(j) \leq C \omega_{\alpha_1,G}(i),
\end{equation} for all indices $i$ and $j$ such that $Q_{i}^{\alpha_1}(G) \cap Q_{j}^{\alpha_2}(\mathbb{R}^n) \neq \emptyset$. Let us begin with the Besov case: \par 
Assume that either $\alpha_1 = 1$ or $\alpha_2 = 1$. The first part of the proof of Theorem \ref{weakly_subordinate_result} regarding exponential versus polynomial growth goes through to show that it is nessesary that both $\alpha_1 = \alpha_2 = 1$. We can now apply Proposition \ref{Besov_are_all_the_same} to obtain that $(\mathbb{R}^n,*_G) \simeq (\mathbb{R}^n,+)$. It remains to show that the smoothness parameters $s_1$ and $s_2$ are equal. However, this is obvious using \eqref{equivalent_weights} and Lemma \ref{lemma_moderate_weights}. Hence all the parameters coincide and $(\mathbb{R}^n,*_G) \simeq (\mathbb{R}^n,+)$. This finishes the Besov case and we can now assume that $0 \leq \alpha_1, \alpha_2 < 1$. \par 
Fix a lattice $N$ in $(\mathbb{R}^n,*_G)$ and assume that either $\alpha_1 = 0$ or $\alpha_2 = 0$. Since the elements in $\mathcal{U}(G;N)$ and $\mathcal{U}(\mathbb{R}^n;\mathbb{Z}^n)$ have constant size, it it is nessesary then that both $\alpha_1 = \alpha_2 = 0$ since the coverings are equivalent. It follows from Proposition \ref{Besov_are_all_the_same} that this forces $(\mathbb{R}^n,*_G) \simeq (\mathbb{R}^n,+)$. Since the choice of lattice is irrelevant, we can choose the lattice $\mathbb{Z}^n$ for both coverings. We then get from \eqref{equivalent_weights} that there exists a $C > 0$ such that \[\frac{1}{C}(1 + \|k\|_{E}^{2})^{\frac{s_1}{2}} \leq (1 + \|k\|_{E}^{2})^{\frac{s_2}{2}} \leq C(1 + \|k\|_{E}^{2})^{\frac{s_1}{2}},\] for all $k \in \mathbb{Z}^{n} \setminus \{0\}$. The equality $s_1 = s_2$ follows from considering $k = (m, 0, \dots, 0) \in \mathbb{Z}^n \setminus \{0\}$ for $m \in \mathbb{N}$. \par
For the intermediate case $0 < \alpha_1, \alpha_2 < 1$ we will first show that $\alpha_1 = \alpha_2$ with a restriction argument. Let $k$ denote the rank of $(\mathbb{R}^n,*_G)$ and let $V_1$ be the first layer in the stratification \eqref{Lie_algebra_decomposition}. Consider the restricted covering on $\mathbb{R}^{k}$ given by \[\mathcal{Q}^{\alpha_1}(G|\mathbb{R}^{k}) := \left(Q_{l}^{\alpha_1}(G) \cap \left(\mathbb{R}^{k} \times \{0\}^{n-k}\right)\right)_{l \in N \setminus \{0\}}.\] To be a bit pedantic, we have defined coverings as consisting of non-empty subsets so we would actually need to remove all the empty sets and renumber the index set $N \setminus \{0\}$ accordingly. However, this will play no role so we omit this insignificant detail. It is straightforward to see that $\mathcal{Q}^{\alpha_1}(G|\mathbb{R}^{k})$ is an admissible covering. Each element in $\mathcal{Q}^{\alpha_1}(G|\mathbb{R}^{k})$ is open and connected due to the subspace topology on $\mathbb{R}^{k}$. It is clear when using the homogeneous quasi-norm $\|\cdot\|_2$ given in \eqref{spesific_homogeneous_quasi_norm} that $\mathcal{Q}^{\alpha_1}(G|\mathbb{R}^{k})$ is an $\alpha_1$-covering on $\mathbb{R}^{k}$. It now follows from Proposition \ref{proposition_about_equivalence} that $\mathcal{Q}^{\alpha_1}(G|\mathbb{R}^{k})$ is equivalent to $\mathcal{Q}^{\alpha_1}(\mathbb{R}^k)$. Since $\mathcal{Q}^{\alpha_1}(G)$ is equivalent to $\mathcal{Q}^{\alpha_2}(\mathbb{R}^n)$ we obtain by restricting that $\mathcal{Q}^{\alpha_1}(\mathbb{R}^k)$ is equivalent to $\mathcal{Q}^{\alpha_2}(\mathbb{R}^k)$. We can now apply Theorem \ref{weakly_subordinate_result} to obtain that $\alpha := \alpha_1 = \alpha_2$. \par 
The next step is to show that the homogeneous dimension $Q$ of $(\mathbb{R}^n,*_G)$ is actually equal to $n$. Fix a lattice on the form $N = \gamma \mathbb{Z}^{k} \times N'$ for $\gamma > 0 $ and use the notation $\vec{l} = (\gamma l, 0, \dots, 0)$ for $l \in \mathbb{N}$. Then by using the homogeneous quasi-norm $\|\cdot\|_2$ we have $Q_{\vec{l}}^{\alpha}(G) \cap Q_{\vec{l}}^{\alpha}(\mathbb{R}^n) \neq \emptyset$ and the estimates
\begin{equation}
\label{estimates_in_main_theorem}
    |Q_{\vec{l}}^{\alpha}(G)| \asymp l^{\frac{Q \alpha}{1 - \alpha}}, \qquad |Q_{\vec{l}}^{\alpha}(\mathbb{R}^n)| \asymp l^{\frac{n\alpha}{1 - \alpha}}, \qquad \frac{|Q_{\vec{l}}^{\alpha}(G)|}{|Q_{\vec{l}}^{\alpha}(\mathbb{R}^n)|} \asymp l^{(Q - n)\frac{\alpha}{1 - \alpha}}.
\end{equation}
Notice that if $Q > n$ then for large $l$ the ratio in \eqref{estimates_in_main_theorem} tends to zero. Recall that the neighbours of $Q_{\vec{l}}^{\alpha}(\mathbb{R}^k)$ are of roughly the same size as $Q_{\vec{l}}^{\alpha}(\mathbb{R}^k)$ because of \eqref{size_of_coverings}. This is a contradiction to the equivalence of the coverings $\mathcal{Q}^{\alpha}(G)$ and $\mathcal{Q}^{\alpha}(\mathbb{R}^n)$. Thus we conclude that $Q = n$ and this implies as previously mentioned that $(\mathbb{R}^n,*_G) \simeq (\mathbb{R}^n,+)$. \par We can now use the standard lattice $\mathbb{Z}^n$ and the standard Euclidean norm $\|\cdot\|_E$ for both coverings. From \eqref{equivalent_weights} there exists a $C > 0$ such that \[\frac{1}{C}\left(1 + \|k\|_{E}^{\frac{2}{1-\alpha}}\right)^{\frac{s_1}{2}} \leq \left(1 + \|k\|_{E}^{\frac{2}{1 - \alpha}}\right)^{\frac{s_2}{2}} \leq C\left(1 + \|k\|_{E}^{\frac{2}{1 - \alpha}}\right)^{\frac{s_1}{2}},\] for all $k \in \mathbb{Z}^{n} \setminus \{0\}$. By again considering $k = (m, 0, \dots, 0) \in \mathbb{Z}^n \setminus \{0\}$ for $m \in \mathbb{N}$ we see that $s_1 = s_2$. Thus all the parameters coincide and $(\mathbb{R}^n,*_G) \simeq (\mathbb{R}^n,+)$.
\end{proof}

\begin{remark}
\begin{itemize}
    \item When the parameters are trivial we have the equality \[M_{2,2}^{0,\alpha_1}(G) = M_{2,2}^{0,\alpha_2}(\mathbb{R}^n) = L^{2}(\mathbb{R}^n),\] where $n := \textrm{dim}(G)$ and $0 \leq \alpha_1, \alpha_2 \leq 1$ by an application of the Pythagorean theorem. 
    \item Usually, we can treat the uniform covering $\mathcal{U}(G;N)$ as a special case of the covering $\mathcal{Q}^{\alpha}(G;N)$ corresponding to $\alpha = 0$. However, a careful inspection of the proof of Theorem \ref{novelty_theorem} shows that the approach in \eqref{estimates_in_main_theorem} breaks down for $\alpha = 0$. This is why we treated the uniform case separately with techniques from metric space geometry by invoking Proposition \ref{Besov_are_all_the_same}. The reader should be aware that we needed to use both metric space geometry arguments and the highly non-trivial results \cite[Theorem 6.9]{Felix_main} and \cite[Theorem 3.7]{Hans_Grobner} to prove Theorem \ref{novelty_theorem}.
\end{itemize}
\end{remark}

\begin{corollary}
Consider a generalized $\alpha_1$-covering $\mathcal{P}^{\alpha_1}(G)$ and a generalized $\alpha_2$-covering $\mathcal{P}^{\alpha_1}(H)$ corresponding to admissible Lie groups $(\mathbb{R}^n,*_G)$ and $(\mathbb{R}^n,*_H)$, respectively. Then $\mathcal{P}^{\alpha_1}(G)$ can only be equivalent to $\mathcal{P}^{\alpha_1}(H)$ whenever $\alpha_1 = \alpha_2$.
\end{corollary}

\begin{proof}
It suffices to consider the explicit coverings $\mathcal{Q}^{\alpha_1}(G)$ and $\mathcal{Q}^{\alpha_2}(H)$ given in \eqref{explicit_coverings_cases} due to Proposition \ref{proposition_about_equivalence}. We have remarked in the proof of Theorem \ref{novelty_theorem} that $\alpha_1 = 0$ implies that $\alpha_2 = 0$ and that $\alpha_1 = 1$ implies that $\alpha_2 = 1$. Hence $0 < \alpha_1,\alpha_2 < 1$ and we can use the restriction trick in the proof of Theorem \ref{novelty_theorem} to reduce both coverings to $\mathbb{R}^k$, where $k := \min\{\textrm{rank}(\mathbb{R}^n,*_G),\textrm{rank}(\mathbb{R}^n,*_H)\}$. Thus $\alpha_1 = \alpha_2$ follows from Theorem \ref{weakly_subordinate_result}.
\end{proof}

\section{Geometric Embeddings Between Generalized Modulation Spaces}
\label{sec: Metric_Geometry_G_alpha_Modulation_Spaces}

Consider two admissible Lie groups $(\mathbb{R}^n,*_G)$ and $(\mathbb{R}^m,*_H)$ together with the spaces $M_{p_1,q_1}^{s_1,\alpha_1}(G)$ and $M_{p_2,q_2}^{s_2,\alpha_2}(H)$. We would like to understand when $M_{p_1,q_1}^{s_1,\alpha_1}(G)$ embeds into the space $M_{p_2,q_2}^{s_2,\alpha_2}(H)$ in a way that preserves the global features of the underlying coverings $\mathcal{Q}^{\alpha_1}(G)$ and $\mathcal{Q}^{\alpha_2}(H)$. When $n = m$, we can simply consider whether the inclusion $M_{p_1,q_1}^{s_1,\alpha_1}(G) \subset M_{p_2,q_2}^{s_2,\alpha_2}(H)$ is bounded. However, when $n \neq m$ we need to be able to compare the coverings $\mathcal{Q}^{\alpha_1}(G)$ and $\mathcal{Q}^{\alpha_2}(H)$ even though they are not on the same space. Hence the commonly used notions of subordinate and weakly subordinate coverings introduced in Subsection \ref{sec: Admissible_coverings} are no longer applicable. However, we see from Proposition \ref{Hans_result} that we should ask that the embedding \[F: M_{p_1,q_1}^{s_1,\alpha_1}(G) \to M_{p_2,q_2}^{s_2,\alpha_2}(H)\] in some way induces a quasi-isometric embedding \[F_{*}:\left(\mathbb{R}^n,d_{\mathcal{Q}^{\alpha_1}(G)}\right) \longrightarrow \left(\mathbb{R}^m,d_{\mathcal{Q}^{\alpha_2}(H)}\right).\] The correct formalization for this was investigated for a very general class of spaces known as \textit{decomposition spaces} in \cite{Berge}. We will briefly review the technical details adapted to our setting. 

\begin{definition}
\label{definition_essential_supprot}
Consider the generalized $\alpha$-modulation space $M_{p,q}^{s,\alpha}(G)$ for some $1 \leq p,q \leq \infty$, $s \in \mathbb{R}$, and $0 \leq \alpha \leq 1$ corresponding to an admissible Lie group $(\mathbb{R}^n,*_G)$. Fix a lattice $N \subset (\mathbb{R}^n,*_G)$. The \textit{essential support} of an element $f \in M_{p,q}^{s,\alpha}(G)$ with respect to the generalized $\alpha$-covering $\mathcal{Q}^{\alpha}(G;N) = (Q_{i}^{\alpha})_{i \in I}$ is defined to be \[\mathcal{C}[f] := \bigcup_{i \in I} \left\{Q_{i}^{\alpha} \, \, \Big | \, \, \|\mathcal{F}^{-1}\left(\psi_{i} \cdot \mathcal{F}(f)\right)\|_{L^{p}} \neq 0 \right\},\] where $(\psi_{i})_{i \in I}$ is any choice of $\mathcal{Q}^{\alpha}(G;N)$-BAPU.
\end{definition}
To clarify, the index set $I$ in Definition \ref{definition_essential_supprot} is equal to $N \setminus \{0\}$ when $0 \leq \alpha < 1$ and equal to $\mathbb{N}_0$ when $\alpha = 1$. Although the essential support of $f \in M_{p,q}^{s,\alpha}(G)$ does depend on the choice of lattice $N$ and the $\mathcal{Q}^{\alpha}(G)$-BAPU, it will be clear that specific choices are irrelevant. The reason we need to utilize this general notion of support is that not every element in $M_{p,q}^{s,\alpha}(G)$ can be realized as a function on $\mathbb{R}^n$; this is already the case for the Euclidean modulation spaces $M_{p,q}(\mathbb{R}^n)$ for most values of $1 \leq p,q \leq \infty$. \par
For every $i \in I$ and $k \in \mathbb{N}_{0}$ we can find an element $g_{i,k} \in M_{p,q}^{s,\alpha}(G)$ such that the essential support of $g_{i,k}$ is contained in $\left(\mathcal{Q}_{i}^{\alpha}\right)^{k*}$. We can even choose the elements $g_{i,k}$ to be smooth functions with compact support since $\mathcal{S}(\mathbb{R}^n)$ is contained in $M_{p,q}^{s,\alpha}(G)$ by Proposition \ref{Schwartz_functions} and there exist smooth $\mathcal{Q}^{\alpha}(G)$-BAPU's with compact support for all $0 \leq \alpha \leq 1$. Hence the following definition is well-defined.

\begin{definition}
\label{geometric_embedding_definition}
Consider the spaces $M_{p_1,q_1}^{s_1,\alpha_1}(G)$ and $M_{p_2,q_2}^{s_2,\alpha_2}(H)$ associated to the admissible Lie groups $(\mathbb{R}^n,*_G)$ and $(\mathbb{R}^n,*_H)$, respectively. We say that a map \[F:M_{p_1,q_1}^{s_1,\alpha_1}(G) \to M_{p_2,q_2}^{s_2,\alpha_2}(H)\] is a \textit{geometric embedding} if $F$ is an injective bounded map between normed spaces with the following additional requirement: There should exist constants $L,C > 0$ such that for any $k \in \mathbb{N}_{0}$ and any $f,g \in M_{p_1,q_1}^{s_1,\alpha_1}(G)$ with $\mathcal{C}[f] \subset (Q_{i}^{\alpha_1})^{k*}$ and $\mathcal{C}[g] \subset (Q_{j}^{\alpha_1})^{k*}$, we have
\begin{equation}
\label{geometric_condition}
 \frac{1}{L}d_{\mathcal{Q}^{\alpha_1}(G)}(x,y) - C \leq d_{\mathcal{Q}^{\alpha_2}(H)}\left(z,w\right) \leq Ld_{\mathcal{Q}^{\alpha_1}(G)}(x,y) + C,
\end{equation}
where $x \in (Q_{i}^{\alpha_1})^{k*}$, $y \in (Q_{j}^{\alpha_1})^{k*}$, $z \in \mathcal{C}[F(f)]$ and $w \in \mathcal{C}[F(g)]$ are arbitrary. The spaces $M_{p_1,q_1}^{s_1,\alpha_1}(G)$ and $M_{p_2,q_2}^{s_2,\alpha_2}(H)$ are said to be \textit{geometrically isomorphic} if there exists an invertible geometric embedding from $M_{p_1,q_1}^{s_1,\alpha_1}(G)$ to $M_{p_2,q_2}^{s_2,\alpha_2}(H)$ whose inverse is also a geometric embedding. 
\end{definition}

\begin{remark}
Notice that we have left out the choice of the lattice in Definition \ref{geometric_embedding_definition} although it implicitly appears in the essential supports and in the distances. The fact that any two lattices in a stratified Lie group are quasi-isometric as metric spaces \cite[Corollary 5.5.9]{Geometric_Group_Theory} implies that the choice of lattices are irrelevant when discussing the existence or non-existence of geometric embeddings. It is also straightforward to see that specific choices of BAPU's does not change the existence question. Hence we can treat existence of geometric embeddings as a canonical property of generalized $\alpha$-modulation spaces.
\end{remark}

It is straightforward to see that a composition of geometric embeddings is again a geometric embedding. The most important property of a geometric embedding $F:M_{p_1,q_1}^{s_1,\alpha_1}(G) \to M_{p_2,q_2}^{s_2,\alpha_2}(H)$ is that it induces a quasi-isometric embedding between the metric spaces $\left(\mathbb{R}^n,d_{\mathcal{Q}^{\alpha_1}(G)}\right)$ and $\left(\mathbb{R}^m, d_{\mathcal{Q}^{\alpha_2}(H)}\right)$ \cite[Proposition 4.6]{Berge}. In our case, this can be described as follows: For $x \in \mathbb{R}^n$ we pick $i \in I$ such that $x \in Q_{i}^{\alpha_1}$ and choose a non-zero function $g_i \in \mathcal{S}(\mathbb{R}^n)$ with $\mathcal{C}[g_i] \subset Q_{i}^{\alpha_1}$. There exists an element $y \in \mathcal{C}[F(g)]$ since $F$ is assumed to be injective. If we define \[F_{*}: \left(\mathbb{R}^n,d_{\mathcal{Q}^{\alpha_1}(G)}\right) \longrightarrow{} \left(\mathbb{R}^m, d_{\mathcal{Q}^{\alpha_2}(H)}\right), \qquad F_{*}(x) = y,\] then $F_{*}$ is easily seen to be a quasi-isometric embedding.\par 
The way to think about geometric embeddings is that they are Banach spaces embeddings that do not \textquote{scramble} the frequency information to much. It might change the frequency information slightly in some bounded region, but we have global control over the displacements. We will focus on the geometric embeddings of the generalized modulation spaces $M_{p,q}^{s}(G)$ with underlying coverings $\mathcal{U}(G)$. The following result was proved in \cite[Theorem 5.2]{Berge} and settles the question for Euclidean modulation spaces.

\begin{proposition}
\label{modulation_space_result}
For $1 \leq p,q < \infty$ there is a tower of compatible geometric embeddings \[M_{p,q}(\mathbb{R}) \xrightarrow{\Gamma_{1}^{2}}M_{p,q}(\mathbb{R}^2) \xrightarrow{\Gamma_{2}^3} \dots \xrightarrow{\Gamma_{n-1}^{n}} M_{p,q}(\mathbb{R}^n) \xrightarrow{\Gamma_{n}^{n+1}} \dots, \] where there are no geometric embeddings in the other direction.
\end{proposition}

While Proposition \ref{modulation_space_result} was proved by using the short-time Fourier transform, this is not available to us and we need to use the stratified structure of our group. The following result can be seen as partly generalizing Proposition \ref{modulation_space_result} to our setting.

\begin{theorem}
\label{new_geometric_embedding_result}
Let $(\mathbb{R}^n,*_G)$ be an admissible Lie group with rank $k$. There exists a geometric embedding \[F: M_{p,q}^{s}(\mathbb{R}^{k'}) \xrightarrow{ } M_{p,q}^{s}(G)\] for every $k' \leq k$, $1 \leq p,q < \infty$ and $s \in \mathbb{R}$. This is optimal in the sense that there does not necessarily exists a geometric embedding from $M_{p,q}^{s}(\mathbb{R}^{l})$ to $M_{p,q}^{s}(G)$ whenever $l > k$.
\end{theorem}

\begin{proof}
It suffices to prove the embedding statement only for $k = k'$. Once this has been shown, the general statement follows from Proposition \ref{modulation_space_result} and the fact that the composition of two geometric embeddings is a geometric embedding. Let us first set the stage by deciding the correct lattice, homogeneous quasi-norm and BAPU. Through the exponential map, we can always find a lattice $N$ such that $N = l\mathbb{Z}^{k} \times N'$, where $l$ is some integer. In the Heisenberg case $\mathbb{H}_3$, we can take $l = 2$. In general however, we only know the existence of an $l \in \mathbb{N}$. We will work with the specific homogeneous quasi-norm $\|\cdot\|_{2}$ given in \eqref{spesific_homogeneous_quasi_norm} and utilize that $\|\cdot\|_{2}$ agree with the usual Euclidean norm on the subspace $\mathbb{R}^{k} \times \{0\}^{n-k}$. Fix a $\mathcal{U}(G;N)$-BAPU $(\psi_{m})_{m \in N \setminus \{0\}}$ such that \[\psi_{m}(x) = \phi_{\overline{m}}(x_1, \dots, x_k) \cdot \psi_{m'}^{'}(x_{k+1}, \dots, x_{n})\] where $\overline{m}$ denotes the projection onto the factor $l\mathbb{Z}^k$ and $(\phi_{\overline{m}})_{\overline{m} \in l \mathbb{Z}^k}$ is a $ \mathcal{U}(\mathbb{R}^k)$-BAPU. For the existence of such a $\mathcal{U}(G;N)$-BAPU, we refer the reader to Example \ref{BAPU_example}. \par
Define the map 
\begin{align*}
    F:\mathcal{S}(\mathbb{R}^k) \subset M_{p,q}^{s}(\mathbb{R}^{k}) & \xrightarrow{ } M_{p,q}^{s}(G) \\ f & \longmapsto F(f)(x) = \mathcal{F}_{n}^{-1}\left(\mathcal{F}_{k}(f)(x_1, \dots , x_k)\cdot \xi(x_{k+1}) \cdots \xi(x_{n})\right),
\end{align*}
where $\xi \in C_{c}^{\infty}(\mathbb{R})$ is a positive bump function supported in $(-1/2,1/2)$ and $\mathcal{F}_{k}$ and $\mathcal{F}_{n}$ denote the Fourier transforms in $k$ and $n$ variables, respectively. It is clear from our choice of homogeneous quasi-norm that the induced map of metric spaces \[F_{*}:\left(\mathbb{R}^k,d_{\mathcal{U}(\mathbb{R}^k)}\right) \xrightarrow{} \left(\mathbb{R}^n,d_{\mathcal{U}(G)}\right)\] can be taken to be the inclusion into the first $k$ coordinates. This is clearly a quasi-isometric embedding since the first $k$-coordinates in $(\mathbb{R}^n,*_G)$ is an abelian subgroup isomorphic to $(\mathbb{R}^k,+)$. We will show boundedness of $F$ on the Schwartz space $\mathcal{S}(\mathbb{R}^k)$ and then use that $\mathcal{S}(\mathbb{R}^k)$ is dense in $M_{p,q}^{s}(\mathbb{R}^k)$ for all $1 \leq p,q < \infty$ and $s \in \mathbb{R}$ \cite[Proposition 11.3.4]{TF_analysis} to obtain boundedness on all of $M_{p,q}^{s}(\mathbb{R}^k)$. \par 
We first compute that 
\begin{align*}
    \mathcal{F}_{n}^{-1}\left (\psi_{m} \cdot \mathcal{F}_{n}\left(F(f)\right) \right ) & = 
    \mathcal{F}_{n}^{-1}\left (\phi_{\overline{m}} \otimes \psi_{m'}^{'} \cdot \mathcal{F}_{k}\left(f \otimes \xi \otimes \cdots \otimes \xi\right) \right ) \\ & =
    \mathcal{F}_{k}^{-1}\left (\phi_{\overline{m}} \cdot  \mathcal{F}_{k}\left(f\right) \right ) \cdot \mathcal{F}_{n-k}^{-1}\left(\psi_{m'}^{'} \cdot \xi \otimes \cdots \otimes \xi\right)
\end{align*}
for every $f \in \mathcal{S}(\mathbb{R}^k)$ and $m \in N \setminus \{0\}$. Due to the support condition on $\xi$, the function $\psi_{m}^{'} \cdot \xi \otimes \cdots \otimes \xi$ is only non-zero whenever $m' = (0,\dots, 0)$. Thus we obtain 
\begin{align*}
\|F(f)\|_{M_{p,q}^{s}(G)}^{q} & = \sum_{m \in N \setminus \{0\}}\left(1 + \|m\|_{2}^{2}\right)^{\frac{qs}{2}}\left\|\mathcal{F}_{n}^{-1}\left (\psi_{m} \cdot \mathcal{F}_{n}\left(F(f)\right) \right  )\right\|_{L_{p}}^{q} \\ & = \sum_{\overline{m} \in l \mathbb{Z}^k}\left(1 + \|\overline{m}\|_{2}^{2}\right)^{\frac{qs}{2}}
\left\|\mathcal{F}_{k}^{-1}\left (\phi_{\overline{m}} \cdot  \mathcal{F}_{k}\left(f\right) \right ) \cdot \mathcal{F}_{n-k}^{-1}\left(\psi_{(0,\dots,0)}^{'} \cdot \xi \otimes \cdots \otimes \xi\right)\right\|_{L_{p}}^{q} \\ & \leq C\sum_{\overline{m} \in l \mathbb{Z}^k}\left(1 + \|\overline{m}\|_{E}^{2}\right)^{\frac{qs}{2}}
\left\|\mathcal{F}_{k}^{-1}\left (\phi_{\overline{m}} \cdot  \mathcal{F}_{k}\left(f\right) \right )\right\|_{L_{p}}^{q} \\ & = C\|f\|_{M_{p,q}^{s}(\mathbb{R}^k)},
\end{align*}
where $\|\cdot\|_{E}$ denotes the Euclidean norm in the coordinates $(x_1, \dots, x_k)$. Hence $F$ is a geometric embedding and the optimally statement follows from Proposition \ref{modulation_space_result}. 
\end{proof}

The following consequence of Theorem \ref{new_geometric_embedding_result} is both aesthetically pleasing and reveals the universality of the Feichtinger algebra on the real line. 

\begin{corollary}
The Feichtinger algebra $\mathcal{S}_{0}(\mathbb{R}) := M_{1,1}^{0,0}(\mathbb{R})$ embeds geometrically into all the standard modulation spaces $M_{p,q}(G)$ for $1 \leq p,q < \infty$ and any admissible Lie group $(\mathbb{R}^n,*_G)$.
\end{corollary}

\begin{proof}
It follows from \cite[Theorem 12.2.2]{TF_analysis} that the inclusion $\mathcal{S}_{0}(\mathbb{R}) \hookrightarrow{} M_{p,q}(\mathbb{R})$ is bounded for every $1 \leq p,q < \infty$. This induces the identity map on the metric spaces level and is hence a geometric embedding. Since every admissible Lie group $(\mathbb{R}^n,*_G)$ have rank greater or equal to one, the result follows from Theorem \ref{new_geometric_embedding_result}.
\end{proof}

In \cite[Theorem 3.6]{Berge} the authors proved that, given two rational stratified Lie groups $(\mathbb{R}^n,*_G)$ and $(\mathbb{R}^n,*_H)$, the metric spaces $(G,d_{\mathcal{U}(G)})$ and $(H,d_{\mathcal{U}(H)})$ are not quasi-isometric unless the growth vectors $\mathfrak{G}(G)$ and $\mathfrak{G}(H)$ are equal. This shows that two generalized modulation spaces $M_{p_1,q_1}^{s_1}(G)$ and $M_{p_2,q_2}^{s_2}(H)$ can only be geometrically isomorphic whenever the growth vectors $\mathfrak{G}(G)$ and $\mathfrak{G}(H)$ are the same. In particular, not only are the Heisenberg modulation spaces $M_{p_1,q_1}^{s_1}(\mathbb{H}_n)$ distinct from $M_{p_2,q_2}^{s_2}(\mathbb{R}^{2n+1})$ for any non-trivial values of the parameters by Theorem \ref{novelty_theorem}, they are also not geometrically isomorphic.

\section{Looking Back and Ahead}
\label{sec: Looking_Back_and_Ahead}

Let us return and comment on the five questions raised in the introduction in the setting of admissible Lie groups:

\begin{enumerate}[label=\arabic*)]
    \item We have seen that the generalized $\alpha$-modulation spaces $M_{p,q}^{s,\alpha}(G)$ have natural coverings associated to them. Moreover, we can choose the explicit coverings $\mathcal{Q}^{\alpha}(G)$ given in \eqref{explicit_coverings_cases} for most purposes. How the coverings $\mathcal{U}(G;N)$ underlying the modulation spaces $M_{p,q}^{s}(G)$ is related to the polynomial growth of the lattice $N$ is further discussed in \cite[Chapter 3]{Berge}. 
    \item We have seen that the elements in $M_{p,q}^{s,\alpha}(G)$ are rather exotic distributions on $\mathbb{R}^n$. However, the containment of the Schwartz functions $\mathcal{S}(\mathbb{R}^n) \subset M_{p,q}^{s,\alpha}(G)$ as well as the explicit coverings $\mathcal{Q}^{\alpha}(G)$ in \eqref{explicit_coverings_cases} make the spaces $M_{p,q}^{s,\alpha}(G)$ more concrete. The coverings $\mathcal{Q}^{\alpha}(G)$ are especially explicit in low dimensions since lattices can be explicitly found and one can use the explicit homogeneous quasi-norm $\|\cdot\|_{2}$ given in \eqref{spesific_homogeneous_quasi_norm}.
    \item As we have mentioned, the modulation spaces on the Heisenberg group $M_{p,q}^{s}(\mathbb{H}_n)$ have been recently studied in \cite{David}. This space have another description through representations of a particular stratified Lie group known as the \textit{Dynin-Folland group}. We refer the reader to \cite{David} for more information on this construction. Also in the Heisenberg case, the Besov spaces $\mathcal{B}_{p,q}^{s}(\mathbb{H}_3)$ are new and very concrete spaces where the non-Euclidean dilations on $\mathbb{R}^3$ can be visualized. Moreover, the Besov coverings $\mathcal{B}(G)$ fits within a previously examined framework as explained in the remark preceding Theorem \ref{novelty_theorem}. 
    \item In contrast with the Euclidean setting, the coverings $\mathcal{Q}^{\alpha}(G)$ are not always almost structured coverings when $(\mathbb{R}^n,*_G)$ is an arbitrary rational stratified Lie group as we showed in Proposition \ref{almost_structured_coverings_result}. Moreover, the methods we have used depend more on geometric considerations (such as growth type) than the more prevalent analytic approach used in the Euclidean setting.
    \item The uniqueness of the generalized $\alpha$-modulation space $M_{p,q}^{s,\alpha}(G)$ was completely settled in Theorem \ref{novelty_theorem}. We showed that the spaces $M_{p,q}^{s,\alpha}(G)$ do form new spaces when the parameters $p,q,s,\alpha$ are non-trivial. 
\end{enumerate}

We hope that we have convinced the reader that the spaces $M_{p,q}^{s,\alpha}(G)$ are worthy of further study. We have avoided the quasi-Banach regime where the integrability parameters $p,q$ are also allowed to take the values $0 < p,q < 1$ to make the exposition less technical. We refer the interested reader to \cite[Chapter 9]{Felix_main} where the Euclidean $\alpha$-modulation spaces $M_{p,q}^{s,\alpha}(\mathbb{R}^n)$ are investigated in the quasi-Banach setting. \par 
The most obvious further work on the topic of generalized $\alpha$-modulation spaces is to prove the existence of BAPU's for the coverings $\mathcal{Q}^{\alpha}(G)$ given in \eqref{explicit_coverings_cases} for an arbitrary rational stratified Lie group $(\mathbb{R}^n,*_G)$. This would remove the slightly artificial restriction of having step less than or equal two. Let us also comment on a few other directions that have not yet been explored. \par 
One of the main advantages of the traditional modulation spaces $M_{p,q}^{s}(\mathbb{R}^n)$ is that they admit a \textit{coorbit description} in the following sense: For $f,g \in L^{2}(\mathbb{R}^n)$ with $g \neq 0$ we define the \textit{short-time Fourier transform} (STFT) of $f$ with respect to the window $g$ to be \[V_{g}f(x,\omega) := \int_{\mathbb{R}^n}f(t)\overline{g(t-x)}e^{-2\pi i t \cdot \omega} \, dt, \quad (x,\omega) \in \mathbb{R}^{2n}.\] One can extend the domain of the STFT to $\mathcal{S}(\mathbb{R}^n) \times \mathcal{S}'(\mathbb{R}^n)$ by duality. If $g \in \mathcal{S}(\mathbb{R}^n) \setminus \{0\}$ we have an alternative description of the space $M_{p,q}^{s}(\mathbb{R}^n)$ for $1 \leq p,q \leq \infty$ and $s \in \mathbb{R}$ as follows: A tempered distribution $f \in \mathcal{S}'(\mathbb{R}^n)$ belongs to the space $M_{p,q}^{s}(\mathbb{R}^n)$ if and only if \[\left(\int_{\mathbb{R}^n}\left(\int_{\mathbb{R}^n}|V_{g}f(x,\omega)|^{p}\left(1 + |x| + |\omega|\right)^{ps}\, dx\right)^{\frac{q}{p}} \, d\omega\right)^{\frac{1}{q}} < \infty.\]
We refer the reader to \cite[Chapter 11]{TF_analysis} for an approach to modulation spaces using the coorbit description. The name coorbit description comes from the fact that the STFT is a manifestation of the unitary representation theory of the Heisenberg group \cite[Chapter 9]{TF_analysis}. This falls within a larger framework developed in \cite{Hans_Grochenig1, Hans_Grochenig2} known as \textit{coorbit theory}. Many properties of the modulation spaces $M_{p,q}^{s}(\mathbb{R}^n)$ are more easily understood through the coorbit description. It would be advantageous to find a coorbit description for the modulation spaces $M_{p,q}^{s}(G)$ where $(\mathbb{R}^n,*_G)$ is any rational stratified Lie group. \par 
We would like to emphasize that the theory we have built for the boundary cases $M_{p,q}^{s}(G)$ and $\mathcal{B}_{p,q}^{s}(G)$ is interesting in itself. One can consider the \textit{Sobolev spaces} $W^{s,p}(G) := B_{p,p}^{s}(G)$ associated to any stratified Lie group $(\mathbb{R}^n,*_G)$. We have from Theorem \ref{novelty_theorem} that the spaces $W^{s,p}(G)$ do not coincide with any of the Euclidean Sobolev spaces $W^{s,p}(\mathbb{R}^n)$ unless $(s,p) = (0,2)$, in which case \[W^{0,2}(G) = W^{0,2}(\mathbb{R}^n) = L^{2}(\mathbb{R}^n).\] In particular, the spaces $H^{k}(G) := W^{k,2}(G)$ for $k \in \mathbb{N}$ are alternatives to the Hilbert space Sobolev spaces $H^{k}(\mathbb{R}^n)$ that permeates PDE theory and nearby disciplines. There are many notions of Sobolev spaces on stratified Lie groups in the literature, and it would be interesting to see how our approach fit in. \par 
The modulation spaces $M_{p,q}^{s}(G)$ have not been considered previously in the literature except on the groups $\mathbb{R}^n$ and $\mathbb{H}_{n}$. In the Euclidean case, the Feichtinger algebra $\mathcal{S}_{0}(\mathbb{R}^n) := M_{1,1}^{0,0}(\mathbb{R}^n)$ has several interesting properties: Every element in $\mathcal{S}_{0}(\mathbb{R}^n)$ is a continuous function and $\mathcal{S}_{0}(\mathbb{R}^n)$ is an algebra under both pointwise multiplication and convolution. Similar questions could be asked for the space $\mathcal{S}_{0}(G) := M_{1,1}^{0,0}(G)$ when $(\mathbb{R}^n,*_G)$ is a rational stratified Lie group. Moreover, it would be interesting to see whether the space $\mathcal{S}_{0}(G)$ satisfies a minimality characterization \cite[Theorem 12.1.8]{TF_analysis} similarly to the Feichtinger algebra $\mathcal{S}_{0}(\mathbb{R}^n)$. Finally, one could ask whether $\mathcal{S}_{0}(G)$ gives rise to a \textit{Banach Gelfand triple} \cite{Banach_Gelfand} \[\mathcal{S}_{0}(G) \xhookrightarrow{} L^{2}(\mathbb{R}^n) \xhookrightarrow{} \left(\mathcal{S}_{0}(G)\right)' \simeq M_{\infty,\infty}^{0}(G),\] for any admissible Lie group $(\mathbb{R}^n,*_G)$. These and many more questions could be illuminating even in a special case such as the \textit{free nilpotent Lie group} $\mathbf{F}_{k,2}$ with step two and rank $k$ whose Lie algebra is defined in \cite[Example 1.5]{Le_Donne_1}. This would generalize most of the known results as $\mathbf{F}_{n,1} = \mathbb{R}^n$ and $\mathbf{F}_{2,2} = \mathbb{H}_3$. We encourage the reader to explore these open questions and build on the work presented.

\Addresses

\end{document}